\documentclass[10pt,a4paper]{article}

\usepackage[spanish, english]{babel}
\usepackage[latin1]{inputenc}
\usepackage{amsfonts}
\usepackage{amsmath,hhline,epsfig,latexsym}
\usepackage{amssymb}
\usepackage{amsthm}
\usepackage{bm}
\usepackage{graphicx,color}
\usepackage{calligra}
\usepackage{mathrsfs}

\usepackage{titlesec}

\titleformat{\paragraph}
{\normalfont\normalsize\bfseries}{\theparagraph}{1em}{}
\titlespacing*{\paragraph}
{0pt}{3.25ex plus 1ex minus .2ex}{1.5ex plus .2ex}

\newcommand{\mat}[1]{\mbox{\boldmath{$#1$}}}

\def\betamod{\chi}
\def\Cvec{\mathbf{C}}
\def\dis{\displaystyle}
\def\Fin{\hfill$\Box$}
\def\Hvec{\mathbf{H}}
\def\Lvec{\mathbf{L}}
\def\Lavec{\mathbf{\Lambda}}
\def\lavec{\mathbf{\lambda}}
\def\bvec{\mathbf{b}}
\def\mvec{\mathbf{m}}

\def\nvec{\mathbf{n}}
\def\Om{\Omega}
\def\om{\omega}
\def\Phivec{\tilde{\mathbf{\Phi}}}
\def\Phivecc{{\mathbf{\Phi}}}
\def\Psivecc{\mathbf{\Psi}}

\def\pvec{\mathbf{p}}

\def\xvec{\mathbf{x}}
\def\avec{\mathbf{a}}
\def\fvec{\mathbf{f}}

\def\0vec{\mathbf{0}}
\def\varvec{\bm{\varphi}}
\def\Vvec{\mathbf{V}}
\def\vvec{\mathbf{v}}

\def\Xvec{\mathbf{X}}
\def\Mvec{\mathbf{M}}
\def\Avec{\mathbf{A}}

\def\Wvec{\mathbf{W}}
\def\Yvec{\mathbf{Y}}
\def\xvec{\mathbf{x}}
\def\yvec{\mathbf{y}}

\def\Zvec{\mathbf{Z}}
\def\zvec{\mathbf{z}}
\def\eps{\varepsilon}

\def\uvec{\mathbf{u}}
\def\wvec{\mathbf{w}}

\newtheorem{theorem}{Theorem}
\newtheorem{proposition}{Proposition}
\newtheorem{remark}{Remark}

\def\bibname{Referencias}
\renewenvironment{thebibliography}[1]
{\footnotesize \section*{\normalsize \bibname}
\list{[\arabic{enumi}]}{\settowidth\labelwidth{[#1]}
\leftmargin\labelwidth
\advance\leftmargin\labelsep
\usecounter{enumi}}
}

\begin{document}

\title{\textbf{On the numerical controllability of the two-dimensional heat, Stokes and Navier-Stokes equations}}

\author{
	\textsc{Enrique FERN\'ANDEZ-CARA}\thanks{Dpto.\ EDAN, Universidad de Sevilla, Aptdo.~1160, 41080~Sevilla, 
	Spain. E-mail: {\tt cara@us.es}. Partially supported by grant MTM2013--41286--P (Spain).}, \
	\textsc{Arnaud M\"UNCH}\thanks{Laboratoire de Math\'{e}matiques, Universit\'{e} Blaise Pascal (Clermont-Ferrand~2), 
	UMR CNRS 6620, Campus des C\'ezeaux,  63177~Aubi\`ere, France. E-mail: {\tt arnaud.munch@math.univ-bpclermont.fr}.}\\
	and 
	\textsc{Diego A. SOUZA}\thanks{ Dpto.\ EDAN, Universidad de Sevilla, Aptdo.~1160, 41080~Sevilla, Spain. E-mail: 
	{\tt desouza@us.es}. Partially supported by grant MTM2013--41286--P (Spain).}
}

\date{}

\maketitle
\begin{abstract}
	The aim of this work is to present some strategies to solve numerically controllability problems for the two-dimensional
	heat equation, the Stokes equations and the Navier-Stokes equations with Dirichlet boundary conditions. The main idea is 
	to adapt the Fursikov-Imanuvilov formulation, see~[A.V. Fursikov, O.Yu. Imanuvilov: {\it Controllability of Evolutions 
	Equations,} Lectures Notes Series, Vol.~34, Seoul National University, 1996]; this approach has been followed recently 
	for the one-dimensional heat equation by the first two authors. More precisely, we minimize over the class of admissible 
	null controls a functional that involves weighted integrals of the state and the control, with weights that blow up near 
	the final time. The associated optimality conditions can be viewed as a differential system in the three variables 
	$x_1$, $x_2$ and $t$ that is second--order in time and fourth--order in space, completed with appropriate boundary
	conditions. We present several mixed formulations of the problems and, then, associated mixed finite element Lagrangian
	approximations that are relatively easy to handle. Finally, we exhibit some numerical experiments.
\end{abstract}

\tableofcontents

%%%%%%%%%%%%%%%%%%%%%%%%%%%%%%%%%%
%%%%%%%%%%%%%%%%%%%%%%%%%%%%%%%%%%
%%%%%%%%%%%%%%%%%%%%%%%%%%%%%%%%%%
%%%%%%%%%   NEW SECTION  %%%%%%%%%%%%%%%
%%%%%%%%%%%%%%%%%%%%%%%%%%%%%%%%%%
%%%%%%%%%%%%%%%%%%%%%%%%%%%%%%%%%%
%%%%%%%%%%%%%%%%%%%%%%%%%%%%%%%%%%

\section{Introduction. The controllability problems}

	Let $\Om\subset \mathbb{R}^2$ be a bounded domain whose boundary $\Gamma:=\partial\Om$ is regular enough. 
	Let $\om\subset \Om$ be a (possibly small) nonempty open subset and assume that 
	$T>0$. We will use the notation $Q_\tau=\Om\times(0,\tau)$, $\Sigma_\tau=\Gamma\times(0,\tau)$, 
	$q_\tau=\om\times(0,\tau)$ and $\nvec=\nvec(\xvec)$ will denote the outward unit normal to $\Om$ at any point 
	$\xvec\in\Gamma$.

	Throughout this paper, $C$ will denote  a generic positive constant (usually depending on 
	$\Om$, $\om$ and $T$) and the bold letters and symbols will stand for vector-valued functions and spaces; for instance 
	$\Lvec^2(\Om)$ is the Hilbert space of the functions $\uvec = (u_1, u_2)$ with $u_1,u_2 \in L^2(\Om)$.

	This paper is concerned with the global null controllability of the heat equation
\begin{equation}\label{heat}
	\left\{
		\begin{array}{lcl}
     			y_t  - \Delta y+ G(\xvec,t)\,y = v1_\om	&\hbox{in}		&	Q_T,		\\
     			\noalign{\smallskip}\dis
			y = 0                                                   &\hbox{on}	&	\Sigma_T, \\
    			\noalign{\smallskip}\dis
			y(\cdot,0) = y_0                			&\hbox{in}		&	\Omega
		\end{array}
	\right.
\end{equation}
	and the Stokes equations
\begin{equation}\label{stokes}
	\left\{
		\begin{array}{lcl}
     			\yvec_t  - \nu\Delta \yvec + \nabla \pi = \vvec1_\om	&\hbox{in}		&	Q_T,		\\
    			\noalign{\smallskip}\dis
			\nabla \cdot \yvec = 0						&\hbox{in}		&	Q_T,		\\
     			\noalign{\smallskip}\dis
			\yvec = \0vec 						          	&\hbox{on}	&	\Sigma_T,	\\
     			\noalign{\smallskip}\dis
			\yvec(\cdot,0) = \yvec_0			 			&\hbox{in}		&	\Omega
		\end{array}
	\right.
\end{equation}
	and the local exact controllability to the trajectories of the Navier-Stokes equations
\begin{equation}\label{N-stokes}
	\left\{
		\begin{array}{lcl}
     			\yvec_t  - \nu\Delta \yvec + (\yvec\cdot\nabla)\yvec + \nabla \pi = \vvec1_\om 	&\hbox{in}		&	Q_T,		\\
    			\noalign{\smallskip}\dis
			\nabla \cdot \yvec = 0      					  				         	&\hbox{in}		&	Q_T,   	\\
     			\noalign{\smallskip}\dis
			\yvec = \0vec 						          			                 	&\hbox{on}	&	\Sigma_T,	\\
     			\noalign{\smallskip}\dis
			\yvec(\cdot,0) = \yvec_0			                                        				&\hbox{in}		&	\Omega.
		\end{array}
	\right.
\end{equation}
	Here, $v = v(\xvec,t)$ and $\vvec = \vvec(\xvec, t)$ stand for the controls (they are assumed to act on $\om$ 
	during the time interval $(0,T)$; the symbol $1_\om$ stands for the characteristic function of $\om$). Moreover, in 
	\eqref{heat}, we assume that  $G \in L^\infty(Q_T)$; in \eqref{stokes} and \eqref{N-stokes}, $\nu>0$.

	Let us first consider the system \eqref{heat}. It is well known that, for any $y_0\in L^2(\Om)$,
	$T > 0$ and $v \in L^2(q_T)$, there exists exactly one solution~$y$ to \eqref{heat}, with
$$
	y \in C^0([0, T];L^2(\Om))\cap L^2(0, T;H_0^1(\Om)).
$$
	The {\it null controllability} problem for \eqref{heat} at time $T$ is the following:
{\it
\begin{quote}
	 For any $y_0 \in L^2(\Om)$ find a control $v \in L^2(q_T)$ such that the associated solution to~\eqref{heat} satisfies
\end{quote}
\begin{equation}\label{null_condition}
	y(\xvec,T) = 0\quad\hbox{in}\quad\Omega.
\end{equation}
}

	The following result is also well known; for a proof, see \cite{FursikovImanuvilov}:
\begin{theorem}\label{th-1}
	The heat equation \eqref{heat} is null-controllable at any time $T>0$.
\end{theorem}

	Let us now consider the systems \eqref{stokes} and \eqref{N-stokes}. Let us recall the definitions of some usual spaces 
	in the context of incompressible fluids:
\[
	\begin{alignedat}{2}
		&\dis\Hvec := \left\{\varvec \in \Lvec^2(\Om) : \nabla \cdot \varvec = 0 \hbox{ in } \Om,
		~\varvec\cdot\nvec=0\hbox{ on }\Gamma\right\},\\
		\noalign{\smallskip}
		&\Vvec :=\dis \left\{\varvec \in \Hvec^1_0(\Om) : \nabla \cdot \varvec = 0\hbox{ in }\Om\right\},\\
		\noalign{\smallskip}
		&U:=\dis \left\{\psi \in H^1(\Om): \int_\Om \psi(\xvec)\,d\xvec=0\right\}.
    	\end{alignedat}
\]

	For any $\yvec_0\in \Hvec,~ T > 0$ and $\vvec \in \Lvec^2(q_T)$, there exists exactly one solution $(\yvec,\pi)$ to the Stokes equations~\eqref{stokes} and
	(since we are in the $2D$ case), also one solution $(\yvec,\pi)$ to the Navier-Stokes equations~\eqref{N-stokes}.
	In both cases
$$
	\yvec \in C^0\left([0, T];\Hvec\right)\cap L^2\left(0, T;\Vvec\right),~\pi \in L^2_{\rm loc}(0,T; U).
$$

	In the context of the Stokes system \eqref{stokes}, the {\it null controllability} problem at time $T$ is the following:
{\it
\begin{quote}
	 For any $\yvec_0 \in \Hvec$ find a control $\vvec \in \Lvec^2(q_T)$ such that the associated solution 
	to~\eqref{stokes} satisfies
\end{quote}
\begin{equation}\label{null_condition_stokes}
	\yvec(\xvec,T) = \0vec\quad\hbox{in}\quad\Omega.
\end{equation}
}

	Again, the following result is well known; for a proof, see \cite{FursikovImanuvilov}:
\begin{theorem}\label{th-2}
	The Stokes system~\eqref{stokes} is null-controllable at any time $T>0$.
\end{theorem}

	Let us recall the concept of {\it exact controllability to the trajectories}. The idea is that,
	even if we cannot reach every element of the state space exactly, we can try to reach (in finite time $T$) any 
	state on any trajectory.
	
	Thus, let $(\overline \yvec,\overline\pi)$ be a solution to the uncontrolled Navier-Stokes equations:
\begin{equation}\label{TRAJEC}
	\left\{
		\begin{array}{lcl}
     			\overline\yvec_t -\nu\Delta\overline\yvec+(\overline\yvec\cdot\nabla)\overline\yvec+\nabla \overline\pi=\0vec
			&\hbox{in}&	Q_T,	           	\\
			\noalign{\smallskip}\dis
    			\nabla \cdot\overline \yvec = 0      					  				         					
			&\hbox{in}&	Q_T, 	  		\\
			\noalign{\smallskip}\dis
     			\overline\yvec = \0vec 						          			                						
			&\hbox{on}&	\Sigma_T,		\\
			\noalign{\smallskip}\dis
     			\overline\yvec(\cdot,0) =\overline\yvec_0			             			                           				
			&\hbox{in}&	\Omega.
		\end{array}
	\right.
\end{equation}	
	We will search for controls $\vvec\in \Lvec^2(q_T)$ such that the associated solutions to \eqref{N-stokes} satisfy
\begin{equation}\label{trajec_condition}
	\yvec(\xvec,T) = \overline \yvec(\xvec,T)\quad\hbox{in}\quad\Omega.
\end{equation}	
	The problem of exact controllability to the trajectories for \eqref{N-stokes} is the following:
{\it
\begin{quote}
	For any $\yvec_0\in \Hvec$ and any trajectory $(\overline\yvec,\overline\pi)$, find a control $\vvec\in \Lvec^2(q_T)$
	such that the associated solution to \eqref{N-stokes} satisfies \eqref{trajec_condition}.
\end{quote}
}

	The following result shows that this problem can be solved at least locally when $\overline\yvec$ is bounded; 
	for a proof, see~\cite{FC-G-P,IMANU}:
	
\begin{theorem}\label{th-3}
	The Navier-Stokes equations \eqref{N-stokes} are locally exact controllable to the trajectories
	$(\overline \yvec,\overline \pi)$ with
\begin{equation}\label{regu_trajec}
	\overline\yvec\in \Lvec^\infty(Q_T),~\overline\yvec(\cdot,0)\in \Vvec.
\end{equation}
	In other words, for any $T>0$ and any solution to $\eqref{TRAJEC}$ satisfying \eqref{regu_trajec},
	there exists $\eps>0$ with the following property: if $\yvec_0\in\Vvec$ and $\|\yvec_0-\overline\yvec(\cdot,0)\|_{\Vvec}\leq\eps$,
	one can find controls $\vvec\in \Lvec^2(q_T)$ such that the associated solutions to \eqref{N-stokes} satisfy 
	\eqref{trajec_condition}.
\end{theorem}

	The aim of this paper is to present efficient strategies for the numerical solution of the previous controllability problems.
	These problems concern the computation of external heat sources or force fields that can be applied in a small part of the working domain and control the whole system at a prescribed positive time.
	There are lots of particular situations where this is highly desired;
	in particular, for theoretical and numerical information on control problems for fluid flows, see~\cite{GlowinskiBook, Gunz-book}.
	
	However, the numerical resolution of control problems as those above is not easy.
	This is due to several reasons:
	
\begin{itemize}

\item In the case of the ``linear'' problems \eqref{heat}--\eqref{null_condition} and~\eqref{stokes}--\eqref{null_condition_stokes}, due to the regularizing effect of the PDEs, the equalities \eqref{null_condition} and~\eqref{null_condition_stokes} can be satisfied only in a very small space and standard numerical approximations of the PDEs are unable to capture this.
	For instance, if we look for a minimal $L^2$ norm null control for \eqref{heat}, we are led by duality to an unconstrained extremal problem in a huge space that cannot be approximated efficiently with usual finite dimensional spaces;
	see however~\cite{boyer, boyercanum12} for detailed comments on this issue.

\item On the other hand, in \eqref{N-stokes}--\eqref{trajec_condition} we find the Navier-Stokes system and, obviously, this adds major difficulties.
	Note that, at present, it is unknown whether or not the exact controllability to the trajectories of \eqref{N-stokes} holds without smallness assumptions even when $\overline\yvec \equiv \0vec$.

\end{itemize}

\begin{remark}\label{rem-boundary}
	{\rm
	In this paper, we will only deal with distributed controls.
	In fact, in the case of the Stokes and Navier-Stokes equations, it would be more appropriate from the viewpoint of applications to consider boundary controls acting on a part of $\Sigma$;
	this will be the subject of a forthcoming paper.
	Note however that, in general terms, a boundary control problem can be re-formulated in the form~\eqref{stokes} or~\eqref{N-stokes} by modifying slightly the domain, choosing $\omega$ appropriately
	(outside the original $\Om$) and then considering the restriction of the controlled state to the original $\Om$.
	\Fin
	}
\end{remark}

\

	The paper is organized as follows.
	
	In Section \ref{sec_heat}, we deal with the numerical null controllability of the heat equation. 
	Following ideas from \cite{FursikovImanuvilov}, we reduce the task to the solution of a boundary-value problem that is fourth-order in space 
	and second-order in time. We present a mixed approximate formulation where we avoid the use of $C^1$ finite elements.
	
	 In Sections \ref{sec_stokes} and \ref{Sec_Navier_Stokes}, we present similar numerical strategies to solve  numerically  the controllability problems considered above for the Stokes and the Navier-Stokes equations.
	 The methods are illustrated with several numerical experiments.
	 
	 Finally, Section~\ref{Sec-final} contains several additional comments.
 
%%%%%%%%%%%%%%%%%%%%%%%%%%%%%%%%%%
%%%%  NEW SECTION 
%%%%%%%%%%%%%%%%%%%%%%%%%%%%%%%%%%
 
\section{A strategy for the computation of null controls for the heat equation}\label{sec_heat}

	In this Section, we will start from a formulation of the null controllability problem for \eqref{heat} introduced 
	and extensively used by Fursikov and Imanuvilov, see~\cite{FursikovImanuvilov}.
	We will present some numerical methods essentially obtained by finite dimensional reduction.
	
	Note that this is not the unique efficient approach.
	The first contribution to the numerical solution of null controllability problems of this kind was due to Carthel, Glowinski and Lions in~\cite{carthel}, using duality arguments.
         However, the resulting problems involve some dual spaces which are very difficult
         (if not impossible) to approximate numerically.
         In~\cite{Labbe}, in the context of approximate controllability, a relaxed observability inequality was given for general semi-discrete (in space) schemes, with the parameter $\eps$ of the order of $\Delta x$.
         The work~\cite{boyer} extends the results in~\cite{Labbe} to the fully discrete situation and proves the convergence towards a semi-discrete control, as the time step $\Delta t$ tends to zero;
         let us also mention~\cite{EV09}, where the authors prove that any controllable parabolic equation, be it discrete or continuous in space, is null-controllable after time discretization through the application of an appropriate filtering of the high frequencies.
         For a comparison of the results furnished by various methods, see~the numerical experiments in~\cite{EFC-AM-sema,EFC-AM-dual,AM-DASmixedheat}.

	Let us fix the notation
\[
	Ly := y_t  - \Delta y+ G(x,t) y, \quad L^*p := - p_t  - \Delta p+ G(x,t) p
\]	
	and let the weights $\rho$, $\betamod$ and $\rho_i$ be given by 
\begin{equation}\label{weights-0}
	\rho(\xvec,t) := e^{\betamod(\xvec) / (T-t)}, \quad \betamod(\xvec):= K_1\left(e^{K_2}-e^{\betamod_0(\xvec)}\right), 
	\quad \rho_i(\xvec,t) := (T-t)^{3/2 - i} \rho(\xvec,t), \quad i=0,1,2,
\end{equation}
	where $K_1$ and $K_2$ are sufficiently large positive constants (depending on $T$) and $\betamod_0=\betamod_0(\xvec)$ 
	is a regular bounded function that is positive in $\Om$, vanishes on $\Gamma$ and satisfies
\[
	|\nabla\betamod_0| > 0 \hbox{ in }\overline{\Om} \setminus \om;
\]
	for a justification of the existence of $\betamod_0$, see~\cite{FursikovImanuvilov}.

	The main idea relies on considering the extremal problem
\begin{equation}\label{variational_F-I_h}
	\left\{
		\begin{array}{llr}
			\dis \hbox{Minimize } \ J(y,v)=  {1\over2}\left(\iint_{Q_T}\rho^2|y|^2\,d\xvec\,dt 
			+ \iint_{q_T}\rho_0^2|v|^2\,d\xvec\,dt\right) \\
			\noalign{\smallskip}
			\dis \hbox{Subject to } \ (y,v)\in \mathcal{H}(y_0,T).\\
		\end{array}
	\right.
\end{equation}

	Here, for any $y_0 \in L^2(\Om)$ and any $T > 0$, the linear manifold $\dis\mathcal{H}(y_0,T)$ is given by
\[
	\mathcal{H}(y_0,T): = \{ (y,v): v\in L^2(q_T),~(y,v)\hbox{ satisfies \eqref{heat} and \eqref{null_condition}}\}.
\]

	We have the following result:
\begin{theorem}\label{th-3_h}
	For any $y_0 \in L^2(\Om)$ and any $T > 0$, there exists exactly one solution to~\eqref{variational_F-I_h}.
\end{theorem}

	This result is a consequence of an appropriate {\it Carleman inequality \rm for the heat equation}.
	
	More precisely, let us introduce the space
\begin{equation}\label{defP0}
	P_0 := \{ p \in C^2(\overline{Q}_T) : p = 0 \ \hbox{on} \ \Sigma_T\}.
\end{equation}
	Then, one has:
\begin{proposition}\label{prop-Carleman_h}
	There exists $C_0$, only depending on $\Om$, $\om$ and $T$, such that the following holds for all $p \in P_0$:
\begin{equation}\label{Carleman-ineq_h}
	\begin{array}{c}
		\dis \iint_{Q_T}  \left[ \rho_2^{-2}(|p_t|^2 + |\Delta p|^2) + \rho_1^{-2} |\nabla p|^2 + \rho_0^{-2} |p|^2 \right]d\xvec\,dt	
		\leq C_0\iint_{Q_T} (\rho^{-2} |L^*p|^2 + \rho_0^{-2} |p|^21_\om)\,d\xvec\,dt.
	\end{array}
\end{equation}
\end{proposition}

	Let us introduce the bilinear form $k(\cdot,\cdot)$, with
\begin{equation}\label{defk}
	k(p,p') := \iint_{Q_T} \!\left(\rho^{-2}L^*p\,L^*p' + 1_\omega \rho_0^{-2}p\,p' \right)d\xvec\,dt \quad \forall p,p'\in P_0.
\end{equation}
	In view of the unique continuation property of the heat equation, $k(\cdot\,,\cdot)$ is a scalar product in $P_0$. 
	Indeed, if $p \in P_0$, $L^*p = 0$~in~$Q_T$,~$p = 0$~on~$\Sigma_T$~and~$p = 0$ in~$q_T$, then we necessarily 
	have $p \equiv 0$.

	Let $P$ be the completion of $P_0$ with respect to this scalar product. Then $P$ is a Hilbert space, the functions $p \in P$ satisfy
\begin{equation}\label{finite-rhs_h}
	\iint_{Q_T} \rho^{-2} |L^*p|^2\,d\xvec\,dt + \iint_{q_T} \rho_0^{-2} |p|^2\,d\xvec\,dt < +\infty
\end{equation}
	and, from Proposition~\ref{prop-Carleman_h} and a standard density argument, we also have \eqref{Carleman-ineq_h} for all $p \in P$.

   	Another consequence of Proposition~\ref{prop-Carleman_h} is that we can characterize the space $P$ as follows:
\begin{equation}\label{defP}
	\dis P=\left\{\,p:p, p_t, \partial_j p, \partial_{jk}p \in L^2(0,T-\delta;L^2(\Om)) \ \forall \delta > 0,\ \eqref{finite-rhs_h}\ 
	\hbox{holds}, \ p = 0 \ \hbox{on} \ \Sigma\,\right\}.
\end{equation}

	In particular, we see that any $p \in P$ satisfies $p \in C^0([0,T-\delta];H_0^1(\Om))$ for all
	$\delta>0$ and, moreover,
\begin{equation}\label{cont-C0_h}
	\| p(\cdot\,,0) \|_{H^1_0} \leq C \, k(p,p)^{1/2} \quad \forall p \in P.
\end{equation}

	The main ideas used in this paper to solve numerically \eqref{variational_F-I_h} rely on the following result\,:

\begin{theorem}\label{th-optimality_h}
	Let the weights $\rho$ and $\rho_0$ be chosen as in Proposition~\ref{prop-Carleman_h}. 
	Let $(y,v)$ be the unique solution~to~\eqref{variational_F-I_h}. Then one has
\begin{equation}\label{optimal-yv-good_h}
	y = \rho^{-2} L^*p, \quad v = -\rho_0^{-2}\bigl.p\bigr|_{q_T},
\end{equation}
	where $p$ is the unique solution to the following variational equality in the Hilbert space $P$:
\begin{equation}\label{pb-p_h}
	\left\{
		\begin{array}{l}
			\dis \iint_{Q_T}\left(\rho^{-2}L^*p\,L^*p' + 1_\omega \rho_0^{-2}p\,p' \right)\,d\xvec\,dt	= \int_\Omega \!y_0(\xvec)\,p'(\xvec,0)\,d\xvec \\
			\noalign{\smallskip}
			\dis \forall~ p' \in P; \ p \in P.
		\end{array}
	\right.
\end{equation}
\end{theorem}

\

	We can interpret \eqref{pb-p_h} as the weak formulation of a boundary-value problem for a 
	PDE that is fourth-order in~$\xvec$ and second-order in~$t$. Indeed, taking ``test functions'' $p' \in P$ first with 
	$p' \in C_0^\infty(Q_T)$, then $p' \in C^2(\overline{\Om} \times (0,T))$ and finally $p' \in C^2(\overline{Q}_T)$, 
	we see easily that $p$ must necessarily satisfy\,:
\begin{equation}\label{eq3_h}
	\left\{
		\begin{array}{lcl}
     			L(\rho^{-2}L^*p) + 1_\om \rho_0^{-2}p = 0 & \text{in}& Q_T,\\ 
			\noalign{\smallskip}
			p = 0, \ \ \rho^{-2}L^*p = 0 & \text{on}& \Sigma_T,\\ 
			\noalign{\smallskip}
			\bigl.\rho^{-2}L^*p\bigr|_{t=0} = y_0,\ \ \bigl.\rho^{-2}L^*p\bigr|_{t=T} = 0 & \text{in}& \Om.
		\end{array}
	\right.
\end{equation}

	By introducing the linear form $\ell_0$, with
\begin{equation}\label{defell0}
	\langle \ell_0,p \rangle := \int_\Omega y_0(\xvec)\,p(\xvec,0)\,d\xvec \quad \forall p \in P,
\end{equation}
	we see from \eqref{cont-C0_h} that $\ell_0$ is continuous and \eqref{pb-p_h} can be rewritten in the form 
\begin{equation}\label{pb-p-s_h}
	\dis k(p,p') = \langle \ell_0,p' \rangle \quad \forall p' \in P; \ p \in P.
\end{equation}

	Let $P_h$ denote a finite dimensional subspace of~$P$. A natural approximation of~\eqref{pb-p-s_h} 
	is the following:
\begin{equation}\label{pb-p-h_h}
	\dis k(p_h,p'_h) = \langle \ell_0,p'_h \rangle \quad \forall p'_h \in P_h; \ p_h \in P_h.
\end{equation}

	Thus, to solve numerically the variational equality \eqref{pb-p-s_h}, it suffices to 
	construct explicitly finite dimensional spaces $P_h \subset P$. Notice however that this is possible but needs some work. 
	The reason is that, if $p \in P_h$, then $\rho^{-1}L^*p = \rho^{-1}(-p_t - \Delta p + G(\xvec,t)\,p)$  and 
	$\rho_0^{-1}p1_{\om}$ must belong to~$L^2(Q_T)$. Consequently, $p_h$ must possess first-order time derivatives and up 
	to second-order spatial derivatives in $L_{\rm loc}^2(Q_T)$.
	Therefore, an approximation based on a standard triangulation 
	of~$Q_T$ requires spaces $P_h$ of functions that must be $C^0$ in~$(\xvec,t)$ and $C^1$ in~$\xvec$ and this can be 
	complex and too expensive. Spaces of this kind are constructed for instance in~\cite{ciarletfem}. For example, 
	good behavior is observed for the so called {\it reduced HTC}, {\it Bell} or {\it Bogner-Fox-Schmidt} finite elements; the reader is referred to
	\cite{EFC-AM-sema,AM-DASmixedheat} for numerical approximations of this kind in the framework of 
	one spatial dimension.
	
	In spite of its complexity, the direct approximation of \eqref{pb-p-h_h} has an advantage: it is possible to adapt
	the standard finite element theory to this framework and deduce strong convergence results for the numerical controls and states.

%%%%%%%%%%%%%%%%%%%%%%%%%%%%%%%%%%
%%%%  NEW SUBSECTION 
%%%%%%%%%%%%%%%%%%%%%%%%%%%%%%%%%%

\subsection{First mixed formulation with modified variables}

	Let us introduce the new variable
\begin{equation}\label{new-var_h}
	z :=L^*p 
\end{equation}
	and let us set $Z :=  L^2(\rho^{-1};Q_T)$. Then $z \in Z$ and $L^*p - z = 0$ (an equality in~$Z$).

	Notice that this identity can also be written in the form
$$
	\iint_{Q_T} \left(z-L^*p \right)\psi\,\,d\xvec\,dt= 0 \quad \forall \psi \in C_0^\infty(Q_T);
$$
	Accordingly, we introduce the following reformulation of~\eqref{pb-p-s_h}:
\begin{equation}\label{pb-p-m_h}
	\left\{
		\begin{array}{l}
			\dis \iint_{Q_T}\left(\rho^{-2} z\,z' + \rho_0^{-2}p\,p'1_\omega \right)d\xvec\,dt  
			+ \iint_{Q_T}\left(z'-L^*p'  \right) \lambda\,\,d\xvec\,dt 
			= \int_\Om y_0(\xvec) \, p'(\xvec,0) \,d\xvec,\\ 
			\noalign{\smallskip}\dis
			\dis \iint_{Q_T}\left( z-L^*p  \right) \lambda'\,d\xvec\,dt  = 0,\\ 
			\noalign{\smallskip}\dis
			\dis \qquad \forall (z',p',\lambda') \in Z\times P \times \Lambda ; \ ((z,p),\lambda) \in Z\times P \times \Lambda,
		\end{array}
	\right.
\end{equation}
	where, $\Lambda := L^2(\rho;Q_T)$.
 
	Notice that $Z$, $P$ and~$\Lambda$ are the appropriate spaces to keep all the terms in \eqref{pb-p-m_h} meaningful.

   	Let us introduce the bilinear forms $\alpha(\cdot\,,\cdot)$ and $\beta(\cdot\,,\cdot)$, with
$$
	\alpha((z,p),(z',p')) := \iint_{Q_T}\left(\rho^{-2}  z\,z' + \rho_0^{-2}p\,p'1_\omega  \right)\,d\xvec\,dt \quad
	\forall (z,p),(z',p')\in Z\times P
$$
	and
$$
	\beta((z,p),\lambda) := \iint_{Q_T}\left[ L^*p - z \right] \lambda\,d\xvec\,dt\forall (z,p)\in Z\times P,\,~\forall
	\lambda\in \Lambda
$$
	and the linear form $\ell: X \mapsto \mathbb{R} $, with
\[
	\langle \ell,(z,p) \rangle := \int_\Om \! y_0(\xvec) \, p(\xvec,0)\,d\xvec\,~\forall (z,p)\in Z\times P.
\]
   	Then, $\alpha(\cdot\,,\cdot)$, $\beta(\cdot\,,\cdot)$ and $\ell$ are well-defined and continuous and \eqref{pb-p-m_h} reads:
\begin{equation}\label{m_h}
	\left\{
		\begin{array}{l}
			\dis \alpha((z,q),(z',p')) + \beta((z',p'),\lambda) = \langle \ell,(z',p') \rangle,		\\ 
			\noalign{\smallskip}\dis
			\dis \beta((z,p),\lambda') = 0,									     	\\ 
			\noalign{\smallskip}\dis
			\dis \qquad \forall (z',p',\lambda') \in Z\times P \times \Lambda ; \ ((z,p),\lambda) \in Z\times P \times \Lambda.
		\end{array}
	\right.
\end{equation}

	This is a mixed formulation of the variational problem~\eqref{variational_F-I_h}. In fact, the following result holds:
\begin{proposition}\label{prop-equiv_h}
   	There exists exactly one solution to~\eqref{m_h}. Furthermore, \eqref{pb-p-s_h} and~\eqref{m_h} are equivalent 
	problems in the following sense:
\begin{enumerate}
	\item 
		If $((z,p),\lambda)$ solves \eqref{m_h}, then $p$ solves \eqref{pb-p-s_h}.
	\item 
		Conversely, if $p$ solves \eqref{pb-p-s_h}, there exists $\lambda \in \Lambda$ such that the triplet $((z,p),\lambda)$, with
		$z := L^*p$ solves~\eqref{m_h}.
\end{enumerate}
\end{proposition}
\begin{proof}
	Let us introduce the space
$$
	V:=\{(z,p)\in Z\times P: \beta((z,p),\lambda) = 0,~ \forall \lambda\in \Lambda\}.
$$

	We will check that
\begin{itemize}
	\item 
		$\alpha(\cdot\,,\cdot)$ is coercive in $V$.
	\item 
		$\beta(\cdot\,,\cdot)$ satisfies the usual ``inf-sup" condition with respect to $Z\times P$ and $\Lambda$.
\end{itemize}

	This will be sufficient to guarantee the existence and uniqueness of a solution to \eqref{m_h}; see for instance~\cite{BrezziFortin, rob-tho}.

	The proofs of the previous assertions are straightforward. Indeed, we first notice that, for any $(z,p)\in V$, $z=L^*p$ and thus
\[
	\begin{array}{lll}
		\dis\alpha((z,p),(z,p))	&=&\dis \iint_{Q_T} \left( \rho^{-2} |z|^2 + \rho_0^{-2}|p|^21_\omega\ \right)\,d\xvec\,dt 	\\
		\noalign{\smallskip}\dis
						&=&\dis {1\over2}\iint_{Q_T} \rho^{-2} |z|^2\,d\xvec\,dt
							+\frac{1}{2}\iint_{Q_T}\rho^{-2}|L^*p|^2\,d\xvec\,dt
							+\iint_{Q_T}\rho_0^{-2}|p|^21_\omega\,d\xvec\,dt							\\
		\noalign{\smallskip}\dis
						&=&\dis \frac{1}{2}\|(z,p)\|^2_{Z\times P}+\frac{1}{2}\iint_{Q_T}\rho_0^{-2}|p|^21_\omega\,d\xvec\,dt	\\
		\noalign{\smallskip}\dis
						&\geq&\dis\frac{1}{2}\|(z,p)\|^2_{Z\times P}.
	\end{array}
\]
	This proves that $\alpha(\cdot\,,\cdot)$ is coercive in $V$. 
	
	On the other hand, for any $\lambda\in \Lambda$ there exists 
	$(z^0,p^0)\in X$ such that 
$$
	\beta((z^0,p^0),\lambda)=\|\lambda\|^2_\Lambda\quad\hbox{and}\quad\|(z^0,p^0)\|_{Z\times P}\leq C \|\lambda\|_\Lambda.
$$

	Indeed, we can take for instance $(z^0,p^0)=(-\rho^{2} \lambda,0)$. Consequently,
$$
	\sup\limits_{(z,p)\in X} \frac{\beta((z,p),\lambda)}{\|(z,p)\|_{Z\times P}}\geq \frac{\beta((z^0,p^0),\lambda)}
	{\|(z^0,p^0)\|_{Z\times P}}\geq \frac{1}{C}\|\lambda\|_\Lambda.
$$ 
	Hence, $\beta(\cdot\,,\cdot)$ certainly satisfies the ``inf-sup" condition in $Z\times P\times \Lambda$.
\end{proof}

	An advantage of \eqref{m_h} with respect to the previous formulation \eqref{pb-p-s_h} is that the solution $((z,p),\lambda)$
	furnishes directly the state-control couple that solves \eqref{variational_F-I_h}. Indeed, it suffices to take
$$
	y=\rho^{-1}z,\quad v=-\rho_0^{-2}p|_{q_T}.
$$
	However, we still find spatial second-order derivatives in the integrals in \eqref{m_h} and, consequently, a finite element
	approximation of \eqref{m_h} still needs $C^1$ in space functions. 

%%%%%%%%%%%%%%%%%%%%%%%%%%%%%%%%%%
%%%%  NEW SUBSECTION 
%%%%%%%%%%%%%%%%%%%%%%%%%%%%%%%%%%

\subsection{Second mixed formulation with modified variables}

	Let us introduce the spaces
\[
	\begin{alignedat}{2}
		\noalign{\smallskip}\dis
		\tilde{P}:=~&\bigg\{\, p:\ \dis \iint_{Q_T}\left[ \rho_2^{-2}| p_t|^2 +\rho_1^{-2}|\nabla  p|^2
					+\rho_0^{-2}|  p|^2\right] d\xvec\,dt< +\infty, \ p\bigr|_{\Sigma_T}=0 \bigg\},		\\					\noalign{\smallskip}\dis
		\tilde{\Lambda}:=~&\bigg\{\, \lambda:\ \dis
					\iint_{Q_T}\left( \rho_2^{2}|\lambda|^2 +\rho_1^{2}|\nabla\lambda|^2\right)d\xvec\,dt< +\infty,\ \lambda\bigr|_{\Sigma_T} = 0 \bigg\},
	\end{alignedat}
\]
	the bilinear forms $\tilde{\alpha}(\cdot\,,\cdot)$ and $\tilde{\beta}(\cdot\,,\cdot)$, with
$$
	\tilde{\alpha}(( z, p),( z', p')) := \iint_{Q_T} \left( \rho^{-2} z\, z' + \rho_0^{-2} p\, p'1_\omega  \right)\,d\xvec\,dt
	\quad\forall (z,p),(z',p') \in Z\times \tilde{P}
$$
	and
\[
	\tilde{\beta}(( z, p), \lambda):=
	\dis\iint_{Q_T} \left[\left( z+ p_t-G(\xvec,t)\, p\right) \lambda-\nabla p\cdot \nabla \lambda\right]\,d\xvec\,dt
	\quad \forall(z,p)\in Z\times\tilde{P},\,~\forall\lambda\in \tilde{\Lambda} \\
\]
	and the linear form $\tilde{\ell}$, with
   $$
\langle \tilde{\ell},( z, p) \rangle: = \int_\Om  y_0(\xvec) \,  p(\xvec,0)\,d\xvec\quad\forall(z,p)\in Z\times\tilde{P}.
   $$
   			
	Then $\tilde{\alpha}(\cdot\,,\cdot)$ and $\tilde{\beta}(\cdot\,,\cdot)$ are well-defined and continuous.
	The linear form $\tilde{\ell}$ is also continuous on $Z\times \tilde{P}$, since the functions in $\tilde{P}$
	satisfy $p\in C^0([0,T-\delta];L^2(\Om))~ \forall \delta>0$ and
$$
	\|p(\cdot,0)\|_{L^2(\Om)}\leq C\left(\iint_{Q_T}(\rho_2^2|p_t|^2+\rho_1^2|\nabla p|^2)\,dx\,dt\right)^{1/2}
	\leq C\|p\|_{\tilde{P}}.
$$

	Let us consider the mixed formulation
\begin{equation}\label{m_hh}
	\left\{
		\begin{array}{l}
			\dis \tilde{\alpha}(( z, p),( z', p')) + \tilde{\beta}(( z', p'), \lambda) = \langle \tilde{\ell},( z', p') \rangle,	\\ 
			\noalign{\smallskip}
			\dis \tilde{\beta}(( z, p), \lambda') = 0,													\\ 
			\noalign{\smallskip}
			\dis \qquad \forall ( z', p', \lambda') \in Z\times\tilde{P}\times\tilde{\Lambda}; \ 
			( z, p, \lambda) \in Z\times\tilde{P}\times\tilde{\Lambda}.
		\end{array}
	\right.
\end{equation}

	Notice that the definitions of $Z$, $\tilde{P}$ and~$\tilde{\Lambda}$ are again the appropriate to keep 
	all the terms in \eqref{m_hh} meaningful.

	It is easy to see that any possible solution to \eqref{m_hh} also solves \eqref{m_h}. Indeed, if $(z,p,\lambda)$ 
	solves~\eqref{m_hh}, then $z=L^*p$ in the sence of $\mathscr{D}'(Q_T)$, whence $p\in P$; thus, the integration by parts with
	respect to the spatial variables is fully justified in $\tilde\beta(z,p,\lambda)$ and $(z,p,\lambda)$ certainly solves 
	\eqref{m_h}.
	
	Consequently, there exists at most one solution
	to \eqref{m_hh}. However, unfortunately, a rigorous proof of the existence of a solution to \eqref{m_hh} is, to our knowledge, unknown.
	In practice, what we would need to prove is that the following ``inf-sup'' condition holds:
$$
	\inf\limits_{ \lambda\in \tilde \Lambda}\sup\limits_{( z, p)\in Z\times \tilde P} 
	{\tilde\beta(( z, p),\lambda)\over \|( z, p)\|_{Z\times \tilde P} \| \lambda\|_{\tilde \Lambda}}>0.
$$ 
	But whether or not this holds is an open question.

%%%%%%%%%%%%%%%%%%%%%%%%%%%%%%%%%%
%%%%  NEW SUBSECTION 
%%%%%%%%%%%%%%%%%%%%%%%%%%%%%%%%%%

\subsection{A reformulation of \eqref{m_hh}}

	It is very convenient from the numerical viewpoint to introduce the following new variables:
\begin{equation}\label{new-var_hh}
	\hat z := \rho^{-1}L^*p, \quad \hat p: = \rho_0^{-1}p.
\end{equation}

	This will serve to improve the conditioning  of the approximations given below. 
	
	The mixed
	problem \eqref{m_hh} can be rewritten in the new variables as follows:
\begin{equation}\label{m_hhh}
	\left\{
		\begin{array}{l}
			\dis \hat{\alpha}((\hat z,\hat p),(\hat z',\hat p')) + \hat{\beta}((\hat z',\hat p'),\hat \lambda) 
			= \langle \hat{\ell},(\hat z',\hat p') \rangle,	\\ 
			\noalign{\smallskip}
			\dis \hat{\beta}((\hat z,\hat p),\hat \lambda') = 0,													\\ 
			\noalign{\smallskip}
			\dis \qquad \forall (\hat z',\hat p',\hat \lambda') \in \hat{Z}\times \hat{P}\times\hat{\Lambda}; \ (\hat z,\hat p,\hat \lambda) \in \hat{Z}\times \hat{P}\times\hat{\Lambda},
		\end{array}
	\right.
\end{equation}
	where\: 
\[
	\begin{alignedat}{2}
				\noalign{\smallskip}\dis
				\hat{Z}:=~&L^2(Q_T),						\\
				\noalign{\smallskip}\dis
				\hat{P}:=~&
				\left\{\, \hat p:\ \dis \iint_{Q_T}
				\left[ (T-t)^4|\hat p_t|^2 +(T-t)^2|\nabla \hat p|^2
				+| \hat p|^2\right] d\xvec\,dt< +\infty,~\hat p\bigr|_{\Sigma_T}=0 \right\},		\\
				\noalign{\smallskip}\dis
				\hat{\Lambda}:=~&\bigg\{\, \hat \lambda:\ \dis
				\iint_{Q_T}\left[ (T-t)^{-1}|\lambda|^2 +(T-t)|\nabla \lambda|^2\right]d\xvec\,dt< +\infty,~
				\hat \lambda\bigr|_{\Sigma_T} = 0 \bigg\}
	\end{alignedat}
\]
	and the bilinear forms $\hat{\alpha}(\cdot\,,\cdot)$ and 
	$\hat{\beta}(\cdot\,,\cdot)$ are given by
$$
	\hat{\alpha}((\hat z,\hat p),(\hat z',\hat p')) := \iint_{Q_T} \left( \hat z\,\hat z' + \hat p\,\hat p'1_\omega  \right)\,d\xvec\,dt \quad \forall \hat{Z}\times\hat{P}
$$
	and % AQUI
\[
	\left\{
	\begin{alignedat}{3}
		\noalign{\smallskip}\dis
		\hat{\beta}((\hat z,\hat p),\hat \lambda):=&\dis\iint_{Q_T} (T-t)^{3/2}\left({\hat p}_t\hat \lambda
		-\nabla \hat p\cdot \nabla \hat \lambda -G \,\hat p\,\hat \lambda\right)\,d\xvec\,dt	\\
		&\dis+\iint_{Q_T}\left[ \hat z+ 2(T-t)^{1/2}\nabla \betamod\cdot\nabla \hat p\right]\hat \lambda\,d\xvec\,dt		\\
		\noalign{\smallskip}\dis
		&\dis+\iint_{Q_T}\left[(T-t)^{1/2}\left(-3/2+\Delta \betamod\right)
		+(T-t)^{-1/2}\left(\betamod+|\nabla\betamod|^2\right)\right]\hat p\, \hat \lambda\,d\xvec\,dt					\\
		\forall(\hat z,\hat p)\in\hat{Z}\times&\hat{P}, \,~\forall\hat\lambda\in\hat{\Lambda}.
	\end{alignedat}
	\right.
\]
	and the linear form $\hat{\ell}: \hat{R}\mapsto \mathbb{R}$ is given by
\[
	\langle \hat{\ell},(\hat z,\hat p) \rangle := \int_\Om  \rho_0(\xvec,0) y_0(\xvec) \,\hat  p(\xvec,0)\,d\xvec \quad
	\forall (\hat z,\hat p)\in\hat{Z}\times\hat{P}.
\]
  
%%%%%%%%%%%%%%%%%%%%%%%%%%%%%%%%%%
%%%%  NEW SUBSECTION 
%%%%%%%%%%%%%%%%%%%%%%%%%%%%%%%%%%

\subsection{A numerical approximation based on Lagrangian finite elements}\label{mixed_heat}

	For simplicity, it will be assumed in the sequel that $\Om$ is a polygonal domain and 
	$\om$ is a polygonal subset of~$\Om$. Let $\mathcal{T}_\kappa$ be a classical 
	$2$-simplex triangulation of $\overline{\Om}$ such that 
	$\overline\om=\bigcup_{F\in\mathcal{T}_\kappa,F\subset\om}F$ and let $\mathcal{P}_\tau$ denote a partition of the time interval $[0,T]$.
	Here, $\kappa$ and $\tau$ denote the respective mesh size parameters. We will use the notation $h := (\kappa,\tau)$ and 
	we will denote by $\mathcal{Q}_h$ the family of all sets of the form
$$
	K = F \times [t_1,t_2], \ \ \hbox{with} \ F \in \mathcal{T}_\kappa, \ [t_1,t_2] \in \mathcal{P}_\tau
$$
	and by $\mathcal{R}_h$ the subfamily of the sets $K = F \times [t_1,t_2]\in \mathcal{Q}_h$,
	such that  $F\subset \om$. We have
$$
	\overline Q_T=\bigcup\limits_{K\in \mathcal{Q}_h} K\quad\hbox{and}
	\quad \overline q_T=\bigcup\limits_{K\in \mathcal{R}_h}K.
$$	

	For any couple of integers $m, n\geq 1$, we will set
$$
	~~~~~~~~~~~~~~~~~~~~~~~~~~\hat Z_h(m,n)= \{\, \hat z_h \in C^0(\overline{Q}_T) : \hat z_h|_K \in  (\mathbb{P}_{m,\xvec} \otimes \mathbb{P}_{n,t})(K) \ \ \forall K \in  {\mathcal Q}_{h} \,\},
$$
$$
	\hat P_h(m,n) = \{\, \hat z_h \in\hat  Z_h(m,n) : \hat z_h = 0 \ \hbox{ on } \ \Sigma_T \,\}
$$
	and
$$
	~~~~~~\hat \Lambda_h(m,n) = \{\, \hat z_h \in\hat  P_h(m,n) : \hat z_h(x,T) = 0 \ \hbox{ in } \ \Om \,\}.
$$
	Here,  $\mathbb{P}_{\ell,\xi}$ denotes the space of polynomial functions of order $\ell$ in the variable $\xi$.
	
	Then, $\hat Z_h(m,n)$, $\hat P_h(m,n)$  and $\hat \Lambda_h(m,n)$ are finite dimensional subspaces of 
	$\hat Z$, $\hat P$ and $\hat \Lambda$, respectively.
	Therefore, for any $m,n,m',n',m'',n'' \geq 1$, we can define the product space
$$
	\hat W_h = \hat W_h(m,n,m',n',m'',n'') :=\hat  Z_h(m,n)\times \hat  P_h(m',n') \times \hat \Lambda_h(m'',n'')
$$
% {\color{red}  very very heavy notations, and useless, because, in the sequel, you only use one choice for $m,n, ....$. I would write simple $Z_h$ instead of $Z_h(m,n)$, ....}
	and the following mixed approximation to~\eqref{m_hhh} makes sense:
\begin{equation}\label{m-h_h}
	\left\{
		\begin{array}{l}
			\dis\hat{\alpha}((\hat z_h,\hat p_h),(\hat z'_h,\hat p'_h)) + \hat{\beta}((\hat z'_h,\hat p'_h),\hat \lambda_h) 
			= \langle \hat{\ell},(\hat z'_h,\hat p'_h) \rangle, \\ 
			\noalign{\smallskip}
			\dis \hat{\beta}((\hat z_h,\hat p_h),\hat \lambda'_h) = 0,\\ 
			\noalign{\smallskip}
			\dis \qquad \forall (\hat z'_h, \hat p'_h,\hat \lambda'_h) \in \hat W_h; \ 
			(\hat z_h,\hat p_h,\hat \lambda_h) \in \hat W_h.
		\end{array}
	\right.
\end{equation}

	Let $n_h=\dim\hat  Z_h(m,n)\times \hat  P_h(m',n'),\, m_h=\dim\hat  \Lambda_h(m'',n'') $ and let the real matrices 
	$\hat A_h\in \mathbb{R}^{n_h,n_h}$, $\hat B_h\in \mathbb{R}^{m_h,n_h}$
	and the vector $\hat L_h\in \mathbb{R}^{n_h}$ be defined by   
\[
	\left\{
		\begin{aligned}
			&\hat \alpha((\hat z_h,\hat p_h),(\hat z'_h,\hat p'_h))=
			 \langle\hat A_h \{(\hat z_h,\hat p_h)\},\{(\hat z'_h,\hat p'_h)\}\rangle_{n_h}
			\quad \forall  (\hat z_h,\hat p_h),(\hat z'_h,\hat p'_h)\in\hat  Z_h(m,n)\times \hat  P_h(m',n'), \\
			&\hat \beta((\hat z_h,\hat p_h),\hat \lambda_h)= 
			\langle\hat B_h \{(\hat z_h,\hat p_h)\}, \{\hat \lambda_h\}\rangle_{m_h}
			\quad \forall (\hat z_h,\hat p_h)\in\hat Z_h(m,n)\times \hat  P_h(m',n'), \quad \forall \hat \lambda_h\in \hat \Lambda_h,\\
			&\hat\ell(\hat z_h,\hat p_h)=\langle\hat L_h,\{\hat z_h,\hat p_h\}\rangle _{n_h} 
			\quad \forall (\hat z_h,\hat p_h)\in\hat  Z_h(m,n)\times \hat  P_h(m',n'),
		\end{aligned}
	\right.
\]
	where $\{(\hat z_h,\hat p_h)\}$, $\{(\hat z_h',\hat p_h')\}$, etc. denote the vectors associated to the functions 
	$(\hat z_h,\hat p_h)$, etc. and $\langle\cdot,\cdot\rangle_{n_h}$ is the usual scalar product in~$\mathbb{R}^{n_h}$. 
	With this notation, the problem \eqref{m-h_h} reads as follows: find 
	$ \{(\hat z_h,\hat p_h)\}\in \mathbb{R}^{n_h}$ and $\{\hat \lambda_h\}\in \mathbb{R}^{m_h}$ such that 

\begin{equation} \label{matrixmfh}
	\left(
		\begin{array}{cc}
			\hat A_h & \hat  B_h^T \\
			\hat B_h & 0   
		\end{array}
	\right)
	\left(
		\begin{array}{c}
			\{(\hat z_h,\hat p_h)\}   \\
			\{\hat \lambda_h\}    
		\end{array}
	\right)  =
	\left(
		\begin{array}{c}
			\hat L_h \\
			0   
		\end{array}
	\right).
\end{equation}
	
	The matrix $\hat A_h$ is symmetric and positive semidefinite but not positive definite for any $h>0$. 
	Indeed, one has 
$$
	\langle\hat A_h \{(\hat z_h,\hat p_h)\},\{(\hat z'_h,\hat p'_h)\}\rangle_{n_h}=\iint_{Q_T}
	(|\hat z_h|^2+|\hat p_h|^21_\om)\,dx\,dt
$$
	for all $(\hat z_h,\hat p_h)\in \hat{Z}_h(m,n)\times \hat{P}_h(m',n')$, whence this quantity is zero if
	$\hat p_h=0$ in $q_T$. The matrix of order 
	$m_h+n_h$ in \eqref{matrixmfh} is symmetric but it is unknown if it is singular or not. However, since our main interest
	is to obtain a numerical solution to \eqref{m_hhh}, we will apply a ``reasonable" method to \eqref{matrixmfh}
	with the hope to get good results. In view of the previous assertions, it seems appropriate  to use 
	 an iterative algorithm like for instance the 
	{\it Arrow--Hurwicz method} (for completeness, we will describe this method in the following Section).

%%%%%%%%%%%%%%%%%%%%%%%%%%%%%%%%%%
%%%%  NEW SUBSECTION 
%%%%%%%%%%%%%%%%%%%%%%%%%%%%%%%%%%

\subsection{The Arrow-Hurwicz algorithm}\label{Num-H-2M}

	As already mentioned, it seems convenient
	to solve \eqref{matrixmfh} using an iterative method. Among other possibilities, we have checked that a good choice is the
	so called {\it Arrow-Hurwicz algorithm}. It is the following:
	
\vspace{0.3cm}
\noindent\textbf{ALG (Arrow-Hurwicz):}
\begin{enumerate}
	\item [(i)]  {\it Initialize}\\
		Fix $r,s > 0$. Let  $(\hat z^{(0)}_h,\hat p^{(0)}_h,\hat \lambda^{(0)}_h)$ 
		be arbitrarily chosen in $\hat W_h$.		
		Take, $(\hat z^{(0)}_h,\hat p^{(0)}_h) = (0,0)$ and $\hat \lambda^{(0)}_h = 0$.
		
	 	For $k \geq 0$, assume that $(\hat z^{(k)}_h,\hat p^{(k)}_h)$ and $\hat \lambda^{(k)}_h$ are known. Then: 		
	\item [(ii)] {\it Advance for $(\hat z_h,\hat p_h)$}: Let $(\hat z^{(k+1)}_h,\hat p^{(k+1)}_h)$ be defined by
		$$
			(\hat z^{(k+1)}_h,\hat p^{(k+1)}_h)= (\hat z^{(k)}_h,\hat p^{(k)}_h) 
			- r\left[\,\hat A_h(\hat z^{(k)}_h,\hat p^{(k)}_h)-\hat L_h+\hat B^T_h\hat \lambda^{(k)}_h\right].
		$$
	\item [(iii)] {\it Advance for $\hat \lambda_h$}: Let $\hat \lambda^{(k+1)}_h$ be defined by
		$$
			\hat \lambda^{(k+1)}_h = \hat \lambda^{(k)}_h + rs\hat  B_h(\hat z^{(k+1)}_h,\hat p^{(k+1)}_h).
		$$
	
	Check convergence. If the stopping test is not satisfied, replace $k$ by $k + 1$ and return to step (ii).
\end{enumerate}
\begin{remark}
	{\rm
	The best choice of the parameters $r$~and~$s$~is determined by the smallest and greatest eigenvalues 
	associated to some operators involving the matrix $\hat A_h$ and $\hat B_h$; see for example $\cite{ciarlet_lions,queck,rob-tho}$.
	The main advantage of $\textbf{ ALG 1}$ with respect to other (iterative or not) algorithms is that we do not have to 
	invert in practice any matrix. In the present context, everything works even if $\hat A_h$ is (as we have already said)
	positive semidefinite but not positive definite. The drawback is that we have to find
	good values of $r$ and $s$ and, obviously, this needs some extra work. \Fin
	}
\end{remark}

%%%%%%%%%%%%%%%%%%%%%%%%%%%%%%%%%%
%%%%  NEW SUBSECTION 
%%%%%%%%%%%%%%%%%%%%%%%%%%%%%%%%%%

\subsection{A numerical experiment}\label{Num-H-2M-2}

	We present now some numerical results. From $(\hat z_h,\hat p_h)$, we obtain an approximation of the control by setting 
	$v_h=-\rho_0^{-1}\hat p_h\,1_\om$. The corresponding controlled state $y_h$ can be computed by solving the equation in~\eqref{heat}
	with standard techniques, for instance using the {\it Crank-Nicolson method}. Since the state is directly given 
	by $\rho^{-1}\hat z$, we simply take $y_h = \rho^{-1}\hat z_h$.

	We present in this Section an experiment concerning the numerical solution of \eqref{m_hhh}. The computations have 
	been performed with {\it Freefem++}, see \cite{hecht}. We have used $P_2$-Lagrange finite elements in $(\xvec,t)$ for 
	all the variables $\hat p$, $\hat z$ and $\hat \lambda$. We have taken $\Omega=(0,L_1)\times(0,L_2)$, with 
	$L_1=L_2=1$. For any $(a,b)\in\Om$, we have considered the function $\betamod_0^{(a,b)}$, where\vspace{-0.4cm}
\[
	\betamod_0^{(a,b)}(x_1,x_2)={x_1(L_1-x_1)x_2(L_2-x_2)e^{-[(x_1-c_a)^2+(x_2-c_b)^2]}\over
					      a(L_1-a)b(L_2-b) e^{-[(a-c_a)^2+(b-c_b)^2]}},
\]
\vspace{-0.3cm}
$$
	 c_a=a-{L_1-2a\over2a(L_1-a)},\,~ c_b=b-{L_2-2b\over2b(L_2-b)}.
$$
	Then, if $(a,b)$ belongs to $\om$, the function $\betamod_0^{(a,b)}$ satisfies the conditions in \eqref{weights-0}.
	We have taken $T=1$, $\omega=(0.2,0.6)\times(0.2,0.6)$, $G(\xvec,t):\equiv1$, $K_1=1$, $K_2=2$, $(a,b)=(0.5,0.5)$ 
	and $y_0(\xvec)\equiv1000$. In view of the regularizing effect of the heat equation, the lack of compatibility of the 
	initial and boundary data does not have serious consequences.
	Indeed, it is seen below that the boundary conditions are satisfied as soon as~$t > 0$.
	
	The computational domain and the mesh are shown in Fig.~\ref{mesh_heat_lin}. With these data, the behavior of the 
	Arrow-Hurwicz algorithm is depicted in Table~\ref{Tab-AH-H}, where the first and second relative errors are respectively 
	given by
$$
	{ \|(\hat z^{(k+1)}_h,\hat p^{(k+1)}_h) - (\hat z^{(k)}_h,\hat p^{(k)}_h)\|_{L^2(Q_T)} \over \|(\hat z^{(k+1)}_h,\hat p^{(k+1)}_h)\|_{L^2(Q_T)} }
$$
	and
$$
	{ \| \hat\lambda^{(k+1)}_h - \hat \lambda^{(k)}_h \|_{L^2(Q_T)} \over \| \hat\lambda^{(k+1)}_h \|_{L^2(Q_T)} }\,.
$$
	
	Some illustrative views of the numerical approximations of the control and the state can be found in 
	Fig.~\ref{control_state_heat0_lin_1}-\ref{control_state_heat0_lin_3}.
	
%%%% FIGURE %%%

\begin{figure}[http]
\begin{center}
\begin{minipage}{0.5\textwidth}
\includegraphics[width=\textwidth]{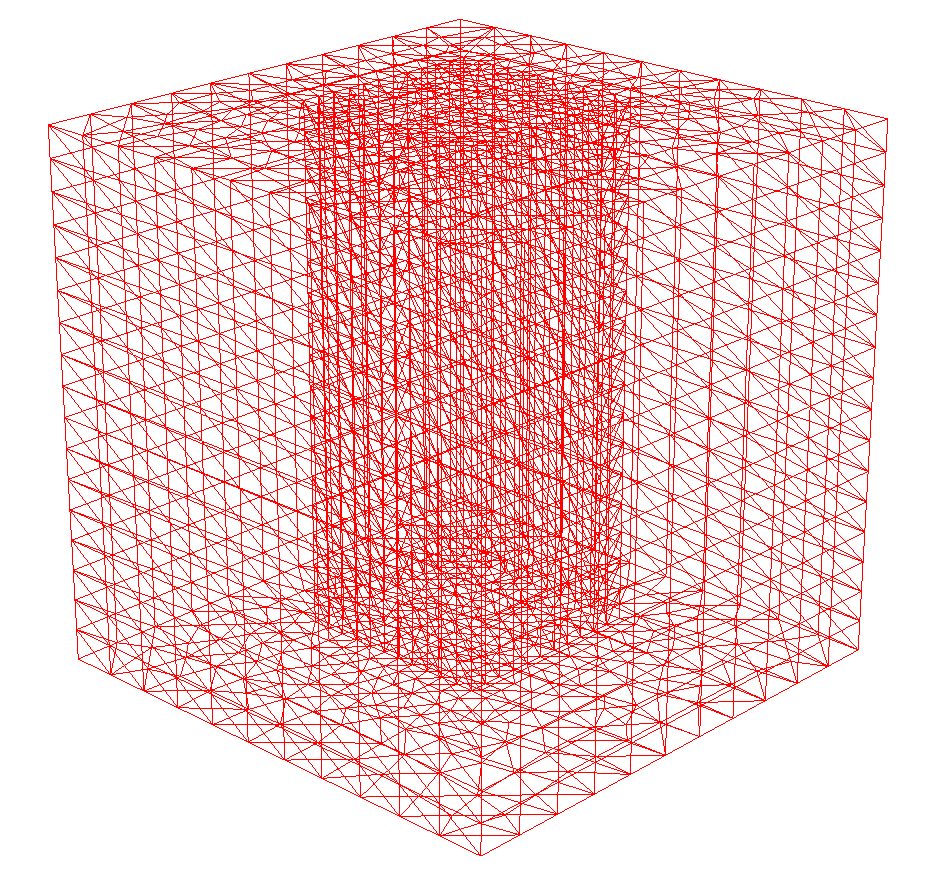}
\end{minipage}
\caption{The domain and the mesh. Number of vertices: 2\ 800. Number of elements (tetrahedra): 14\ 094.
Total number of variables: 20\ 539.}
\label{mesh_heat_lin}
\end{center}
\end{figure}

%%%% FIGURE %%%

%\vskip-2cm
\begin{figure}[http]
\begin{center}
\begin{minipage}{0.49\textwidth}
\includegraphics[width=\textwidth]{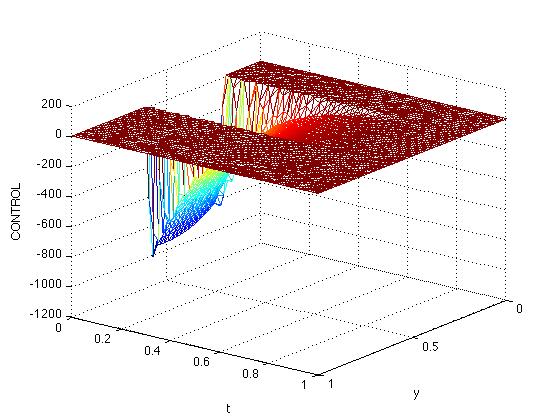}
%\vskip-4cm
\includegraphics[width=\textwidth]{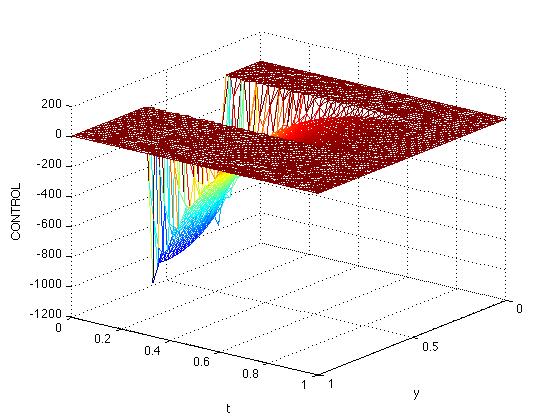}
\end{minipage}
\begin{minipage}{0.49\textwidth}
\includegraphics[width=\textwidth]{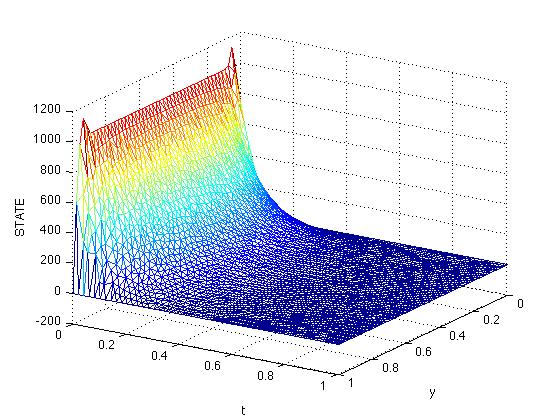}
%\vskip-4cm
\includegraphics[width=\textwidth]{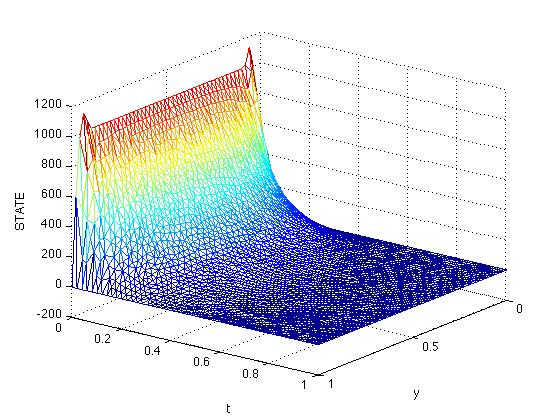}
\end{minipage}
\caption{$\omega=(0.2,0.6)$; $y_0(\xvec)=1000$. Cuts at $x_1=0.28$ and~$x_1=0.52$ of the control $v_h$ ({\bf Left}) and the state ({\bf Right}).}
\label{control_state_heat0_lin_1}
\end{center}
\end{figure}

%%%% FIGURE %%%

%\vskip-2cm
\begin{figure}[http]
\begin{center}
\begin{minipage}{0.49\textwidth}
\includegraphics[width=\textwidth]{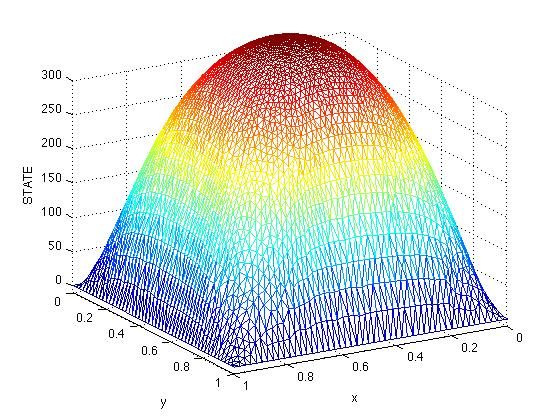}
%\vskip-4cm
\includegraphics[width=\textwidth]{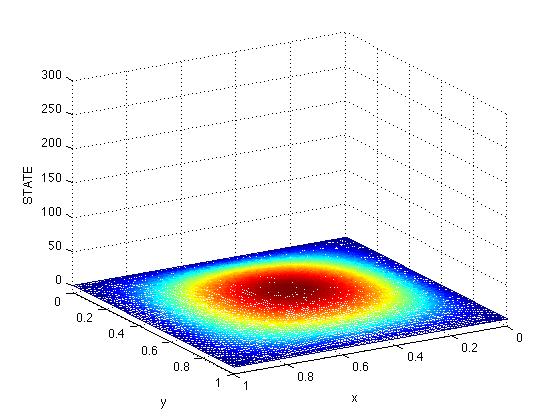}
\end{minipage}
\begin{minipage}{0.49\textwidth}
\includegraphics[width=\textwidth]{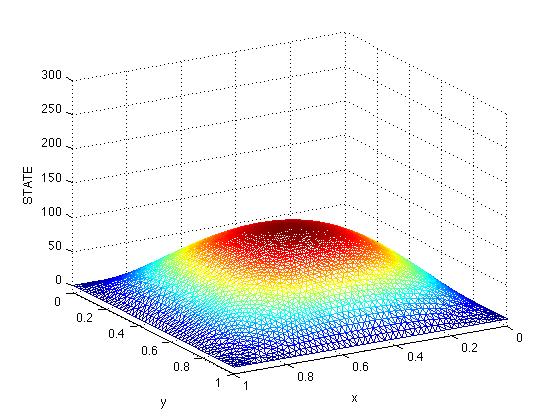}
%\vskip-4cm
\includegraphics[width=\textwidth]{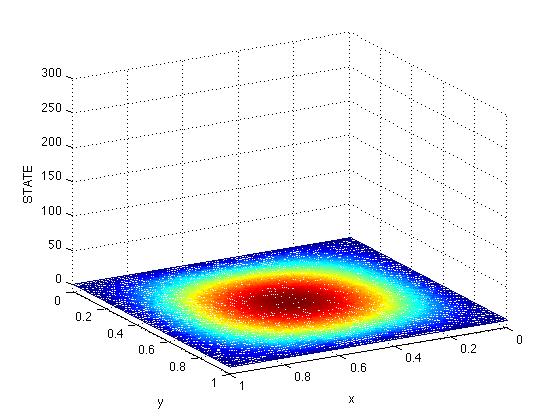}
\end{minipage}
\caption{Evolution of the state: $t=0.2$ and $t=0.6$ ({\bf Left}), $t=0.4$ and $t=0.8$ ({\bf Right}).}
\label{control_state_heat0_lin_2}
\end{center}
\end{figure}

%%%% FIGURE %%%

\begin{figure}[http]
\begin{center}
\begin{minipage}{0.6\textwidth}
\includegraphics[width=\textwidth]{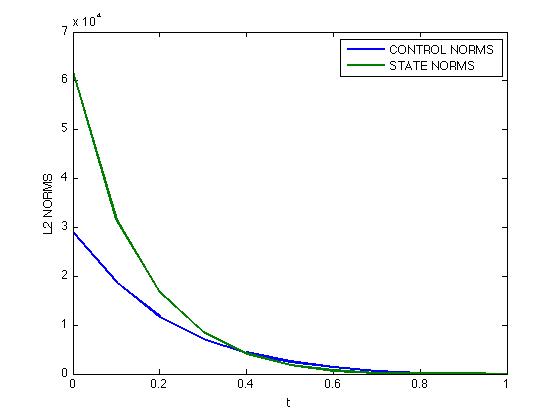}
\end{minipage}
\caption{Evolution of the $L^2$ norms of the control and the state.}
\label{control_state_heat0_lin_3}
\end{center}
\end{figure}

%%%% TABLE %%%
 
\begin{table}[http]
\centering
\begin{tabular}{|c|c|c|}
\hline
Iterate  & Rel.\ error 1 & Rel.\ error 2 \tabularnewline
\hline 
1             & $0.225333$      &  $1.000000$  \tabularnewline\hline 
10           & $0.019314$      &  $0.236549$  \tabularnewline\hline 
20           & $0.010068$      &  $0.065812$  \tabularnewline\hline 
30           & $0.005633$      &  $0.046212$  \tabularnewline\hline 
40           & $0.000358$      &  $0.003397$  \tabularnewline\hline 
50           & $0.000117$      &  $0.001357$  \tabularnewline\hline 
\hline
\end{tabular}
\caption{The behavior of ALG $1$ for \eqref{m_hhh}.}\label{Tab-AH-H}
\end{table}

%%%%%%%%%%%%%%%%%%%%%%%%%%%%%%%%%%
%%%%  NEW SECTION 
%%%%%%%%%%%%%%%%%%%%%%%%%%%%%%%%%%

\section{A strategy for the computation of null controls for the Stokes equations}\label{sec_stokes}
	
   	In this Section, we will present a formulation of the null controllability problem for \eqref{stokes} inspired by the same ideas
	(again, Fursikov-Imanunuvilov's formulation). Specifically, we will try to solve numerically the problem
\begin{equation}\label{variational_F-I}
	\left\{
		\begin{array}{llr}
			\dis \hbox{Minimize } \ J(\yvec,\vvec)=  {1\over2}\iint_{Q_T}  \rho^2|\yvec|^2\,d\xvec\,dt 
			+ {1\over2}\iint_{q_T}\rho_0^2|\vvec|^2 \,d\xvec\,dt\\
			\noalign{\smallskip}
			\dis \hbox{Subject to } \ (\yvec,\vvec)\in \mathcal{S}(\yvec_0,T),
		\end{array}
	\right.
\end{equation}
   	where $\yvec_0 \in \Hvec$, $T > 0$, the linear manifold $\dis\mathcal{S}(\yvec_0,T)$ is given by
\[
	\dis\mathcal{S}(\yvec_0,T) = \left\{\, (\yvec,\vvec): \vvec\in \Lvec^2(\omega\times(0,T)),~(\yvec,\vvec)
	~\hbox{satisfies \eqref{stokes} for some}~\pi ~\hbox{and fulfills}~\eqref{null_condition_stokes} \,\right\}
\]
	and it is again assumed that the weights $\rho$ and $\rho_0$ satisfy~\eqref{weights-0}.

   	We have:
\begin{theorem}\label{th-2_s}
   	For any $\yvec_0 \in \Hvec$ and $T > 0$, there exists exactly one solution to~\eqref{variational_F-I}.
\end{theorem}
	
	Again, this result can be viewed as a consequence of a {\it Carleman inequality}. Thus,
	let us set  
\[
	\Lvec\yvec := \yvec_t -\nu\Delta \yvec, \quad \Lvec^*\pvec := -\pvec_t -\nu\Delta \pvec
\]
	and let us introduce the space
\begin{equation}\label{defPhi0}
	\dis \Phivecc_0 = \{\, (\pvec,\sigma) : p_i, \sigma \in C^2(\overline{Q}_T), \ \nabla\cdot\pvec \equiv0,
	\dis \ p_i = 0 \ \hbox{on} \ \Sigma_T, \ \int_\Om \!\sigma(\xvec,t) d\xvec\!=\! 0 \,\ \forall t \,\}.
\end{equation}	
	Then, one has the following (see \cite{FursikovImanuvilov,IMANU}):
\begin{proposition}\label{prop-Carleman-Stokes} % AQUI
	The function $\betamod_0$ and the associated weights  $\rho$, $\rho_0$, $\rho_1$ and $\rho_2$ can be chosen such that there exists $C$, only depending on $\Om$, $\om$ 
	and $T$,  with the following property:
\begin{equation}\label{Carleman-ineq_stokes}
	\begin{array}{c}
		\dis \iint_{Q_T}  \left[ \rho_2^{-2}(|\pvec_t|^2 + |\Delta\pvec|^2) + \rho_1^{-2} |\nabla\pvec|^2 + \rho_0^{-2} |\pvec|^2 
		+ \rho^{-2} |\nabla\sigma|^2 \right]d\xvec\,dt\\ 
		\noalign{\smallskip}
		\dis \leq C \left( \iint_{Q_T} \rho^{-2} |\Lvec^*\pvec + \nabla\sigma|^2\,d\xvec\,dt
		+ \iint_{q_T} \rho_0^{-2} |\pvec|^2 \,d\xvec\,dt\right)
	\end{array}
\end{equation}
	for all $(\pvec,\sigma)\in \Phivecc_0$.
\end{proposition}
	
   	Let us introduce the bilinear form
\begin{equation}\label{defm}
	\mvec((\pvec,\sigma),(\pvec',\sigma')) := \iint_{Q_T }\left[\rho^{-2}(\Lvec^*\pvec+\nabla\sigma)
	\cdot(\Lvec^*\pvec'+\nabla\sigma') + 1_\omega \rho_0^{-2}\pvec\cdot\pvec' \right]d\xvec\,dt.
\end{equation}
	In view of the unique continuation property of the Stokes system, $\mvec(\cdot\,,\cdot)$ is a scalar product in~$\Phivecc_0$: 
	if $(\pvec,\sigma)\in \Phivecc_0$, $\Lvec^*\pvec + \nabla\sigma = \0vec$ in~$Q_T$ and $\pvec = \0vec$ in~$q_T$, then 
	we have $\pvec \equiv \0vec$ and $\sigma \equiv 0$ (note that, in fact, under these circumstances, $\pvec(\cdot\,,t)$ 
	and $\sigma(\cdot\,,t)$ are analytic for all $t$).

	Let $\Phivecc$ be the completion of $\Phivecc_0$ with respect to this scalar product. As before, $\Phivecc$ is a Hilbert space, 
	the functions $(\pvec,\sigma) \in \Phivecc$ satisfy
\begin{equation}\label{finite-rhs}
	\iint_{Q_T} \rho^{-2} |\Lvec^*\pvec + \nabla\sigma|^2 \,d\xvec\,dt+ \iint_{q_T} \rho_0^{-2} |\pvec|^2\,d\xvec\,dt < +\infty
\end{equation}
	and, from Proposition~\ref{prop-Carleman-Stokes} and a density argument, we also have \eqref{finite-rhs} for any $(\pvec,\sigma) \in \Phivecc$.
   	
	We also see from Proposition~\ref{prop-Carleman-Stokes} that
\begin{equation}\label{defPhi}
	\begin{array}{c}
		\dis \Phivecc = \{\, (\pvec,\sigma) : p_i, \sigma, \partial_tp_i, \partial_{x_j} p_i, 
		\partial_{x_j}\sigma, \partial_{x_jx_k} p_i 
		\in L^2(0,T-\delta;L^2(\Om)) \ \forall \delta > 0,												    \\ 
		\noalign{\smallskip}
		\dis \eqref{finite-rhs}\ \hbox{holds},\ \nabla\cdot\pvec \equiv 0\ \hbox{in} \ Q_T, \ p_i = 0 
		\ \hbox{on} \ \Sigma_T, \ \int_\Om \sigma(\xvec,t)\, d\xvec = 0 \,\ \forall t \,\}
	\end{array}
\end{equation}
	and, in particular, any $(\pvec,\sigma) \in \Phivecc$ satisfies $\pvec \in C^0([0,T-\delta];\Vvec)$ for all $\delta > 0$ and
\begin{equation}\label{cont-C0}
	\|\pvec(\cdot\,,0) \|_{\Vvec} \leq C \, \mvec((\pvec,\sigma),(\pvec,\sigma))^{1/2} \quad \forall (\pvec,\sigma) \in \Phivecc.
\end{equation}
	
	The following result is proved in~\cite{FursikovImanuvilov}:
\begin{theorem}\label{th-optimality}
	Let the weights $\rho$ and $\rho_0$ be chosen as in Proposition~\ref{prop-Carleman-Stokes}. Let $(\yvec,\vvec)$ be the 
	unique solution to~\eqref{variational_F-I}. Then one has  
\begin{equation}\label{optimal-yv-good}
	\yvec = \rho^{-2} (L^*\pvec + \nabla \sigma), \quad \vvec = -\rho_0^{-2}\bigl.\pvec\bigr|_{\om \times (0,T)},
\end{equation}
	where $(\pvec,\sigma)$ is the solution to the following variational equality in $\Phivecc$:
\begin{equation}\label{pb-psigma}
	\left\{
		\begin{array}{l}
			\dis \iint_{Q_T}\left(\rho^{-2}(\Lvec^*\pvec+\nabla\sigma)\cdot(\Lvec^*\pvec'+\nabla\sigma') 
			+\rho_0^{-2}\pvec\cdot\pvec'1_\om \right)d\xvec\,dt
			=\dis\int_\Om\yvec_0(\xvec)\cdot\pvec'(\xvec,0)\,d\xvec 								\\
			\noalign{\smallskip}
			\dis \forall (\pvec',\sigma') \in \Phivecc; \ (\pvec,\sigma) \in \Phivecc.
		\end{array}
	\right.
\end{equation}
\end{theorem}

\

	Once more, \eqref{pb-psigma} can be viewed as the weak formulation of a (non-scalar) boundary-value problem for a PDE 
	that is fourth-order in~$\xvec$ and second-order in~$t$. Indeed, arguing as in Section \ref{sec_heat}, we can easily deduce 
	that $(\pvec,\sigma)$ satisfies, together with some $\pi \in\mathcal{D}'(Q_T)$, the following:
\begin{equation}\label{eq3}
	\left\{
		\begin{array}{lcl}
			\Lvec(\rho^{-2}(\Lvec^*\pvec+\nabla\sigma))+ \nabla \pi + 1_\om \rho_0^{-2}\pvec = 0			
			& \text{in}& Q_T,		\\ 
			\noalign{\smallskip}
			\nabla \cdot (\rho^{-2}(\Lvec^*\pvec+\nabla\sigma)) = 0,~\nabla \cdot\pvec=0    				
			& \text{in}& Q_T, 	  	\\ 
			\noalign{\smallskip}
			\pvec = \textbf{0}, \ \ \rho^{-2}(\Lvec^*\pvec+\nabla\sigma)= \0vec						
			& \text{on}& \Sigma_T, 	\\ 
			\noalign{\smallskip}
			\bigl.\rho^{-2}(\Lvec^*\pvec+\nabla\sigma)\bigr|_{t=0} = \yvec_0,\ \ \bigl.\rho^{-2}(\Lvec^*\pvec
			+\nabla\sigma)\bigr|_{t=T} = \0vec & \text{in}& \Om.
		\end{array}
	\right.
\end{equation}

   	By setting
\begin{equation}\label{defB0}
	\langle \ell_0,(\pvec,\sigma) \rangle := \int_\Omega \yvec_0(\xvec)\cdot\mathbf{p}(\xvec,0)\,d\xvec,
\end{equation}
	it is found that \eqref{pb-psigma} can be rewritten in the form
\begin{equation}\label{pb-psigma-s}
	\dis \mvec((\pvec,\sigma),((\pvec',\sigma')) = \langle \ell_0,(\pvec',\sigma') \rangle
	\dis \quad \forall (\pvec',\sigma') \in \Phivecc; \ (\pvec,\sigma) \in \Phivecc.
   \end{equation}

   	Thus, if $\Phivecc_h$ denotes a finite dimensional subspace of~$\Phivecc$, a natural approximation of~\eqref{pb-psigma-s} 
	is the following:
\begin{equation}\label{pb-psigma-h}
	\dis \mvec((\pvec_h,\sigma_h),((\pvec'_h,\sigma'_h)) = \langle \ell_0,(\pvec'_h,\sigma'_h) \rangle
	\dis \quad \forall (\pvec'_h,\sigma'_h) \in \Phivecc_h; \ (\pvec_h,\sigma_h) \in \Phivecc_h.
\end{equation}

	However, the couples $(\pvec,\sigma) \in \Phivecc$ satisfy several properties that make it considerably difficult to construct
	 explicitly finite dimensional spaces $\Phivecc_h \subset \Phivecc$. These are the following:
\begin{itemize}
	\item 	
		As in Section \ref{sec_heat}, since $\rho^{-1}(\Lvec^*\pvec + \nabla\sigma)$ must belong to~$\Lvec^2(Q_T)$
		and $\bigl. \rho_0^{-1}\pvec \bigr|_{q_T}$ must belong to~$\Lvec^2(q_T)$, the $p_i$ must possess first-order 
		time derivatives and up to second-order spatial derivatives in~$L^2_{\rm loc}(Q_T)$. As before, this means that, in practice, 
		the functions in~$\Phivecc_h$ must be $C^0$ in~$(\xvec,t)$ and $C^1$ in~$\xvec$.
	\item 
		We now have $\nabla \cdot \pvec \equiv 0$. It is not simple at all to give explicit expressions of zero 
		(or approximately zero) divergence functions associated to a triangulation of~$Q_T$ with this regularity. 	
\end{itemize}

	The second inconvenient is classical in computational fluid dynamics when one considers incompressible fluids.
	As in many other works, it will be overcome by introducing additional ``pressure-like'' multipliers;
	see Section~\ref{Sec-mixed_S2}.
	On the other hand, the first difficulty will be circumvented as in Section~\ref{sec_heat}, by introducing new variables and associated multipliers and eliminating all the second-order derivati\-ves in the formulation.

	In the following Sections, we will present several mixed problems connected to~\eqref{pb-psigma-s}.
	
	More precisely, in~Sections~\ref{Sec-mixed_S1} and~\ref{Sec-mixed_S4}, we consider mixed formulations where the constraint $\nabla\cdot\pvec \equiv 0$ is preserved.
	Accordingly, we only introduce one additional variable $\zvec$ and one multiplier, related to the identity $\zvec = L^*\pvec + \nabla\sigma$.
	Contrarily, Sections~\ref{Sec-mixed_S2}, \ref{Sec-mixed_S3} and~\ref{Sec-mixed_S5} deal with other different formulations where the zero-divergence condition is not imposed and, therefore, another multiplier appears.

%%%%%%%%%%%%%%%%%%%%%%%%%%%%%%%%%%
%%%%  NEW SUBSECTION 
%%%%%%%%%%%%%%%%%%%%%%%%%%%%%%%%%%

\subsection{A first mixed formulation of \eqref{pb-psigma-s}}\label{Sec-mixed_S1}

	Arguing as in the case of the heat equation and introducing the variable
$$
	\zvec=\Lvec^*\pvec+\nabla \sigma,
$$
	we see that \eqref{pb-psigma-s} is equivalent to:
\begin{equation}\label{reform_pb-psigma-s}
	\left\{
		\begin{array}{l}
			\noalign{\smallskip}\dis 
				\avec((\zvec,\pvec,\sigma),(\zvec',\pvec',\sigma')) + \bvec((\zvec',\pvec',\sigma'),\lavec) 
				= \langle \mat{\ell},(\zvec',\pvec',\sigma') \rangle	,											\\ 
				\noalign{\smallskip}\dis 
				\bvec((\zvec,\pvec,\sigma),\lavec') = 0,											\\ 
				\noalign{\smallskip}\dis 
				\quad \forall ((\zvec',\pvec',\sigma'),\lavec') \in \Zvec \times \Phivecc \times \Lavec ;
				\ ((\zvec,\pvec,\sigma),\lavec) \in \Zvec \times \Phivecc \times \Lavec,
		\end{array}
	\right.
\end{equation}
	where
\[
	\begin{alignedat}{2}
				\noalign{\smallskip}\dis
				\Zvec=\Lvec^2(\rho^{-1};Q_T),						
				\quad\Lavec=\Lvec^2(\rho;Q_T)
	\end{alignedat}
\]
	and the bilinear forms $\avec(\cdot\,,\cdot)$ and $\bvec(\cdot\,,\cdot)$ are given by
$$
	\avec((\zvec,\pvec,\sigma),(\zvec',\pvec',\sigma')) := \iint_{Q_T} \left(\rho^{-2}\zvec\cdot\zvec' 
	+ \rho_0^{-2}\pvec\cdot\pvec'1_\om  \right)\,d\xvec\,dt
$$
	and
\begin{equation}\label{b_mix_s_1}
	\bvec((\zvec,\pvec,\sigma),\lavec):=\dis\iint_{Q_T} \left[\zvec-(\Lvec^*\pvec+\nabla\sigma)\right]\cdot\lavec\,d\xvec\,dt\\
\end{equation}
	and the linear form $\mat{\ell}$ is given by
\[
	\langle \mat{\ell},(\zvec,\pvec,\sigma) \rangle: = \int_\Om  \yvec_0(\xvec) \cdot \pvec(\xvec,0)\,d\xvec.
\]

%%%%%%%%%%%%%%%%%%%%%%%%%%%%%%%%%%
%%%%  NEW SUBSECTION 
%%%%%%%%%%%%%%%%%%%%%%%%%%%%%%%%%%

\subsection{A second mixed formulation of \eqref{pb-psigma-s}}\label{Sec-mixed_S2}

	As we have said, numerical difficulties are found when we try to introduce finite element approximations ($C^1$ in space, $C^0$ in 
	time) of the space $\Phivecc$, where the $(\pvec,\sigma)$ satisfy $\nabla\cdot\pvec \equiv 0$. Accordingly, before approximating, we will reformulate~\eqref{pb-psigma} as a new mixed system involving a multiplier associated to this constraint. 
	
	Let us introduce
\begin{equation}\label{defPhi00}
	\dis \Phivec_0 = \{\, (\pvec,\sigma) : p_i, \sigma \in C^2(\overline{Q}_T),\dis \ p_i = 0 \ \hbox{on} \ \Sigma_T, \ \int_\Om \sigma(\xvec,t) d\xvec= 0 \,\ \forall t \,\}.
\end{equation}
 
 	We have the following Carleman estimates for the couples in $ \Phivec_0$:	
	
\begin{proposition}\label{prop-Carleman-Stokess}
	There exist weights  $\rho$, $\rho_0$ and $\rho_*$ and a constant $C$ only depending on $\Om$, $\om$ and~$T$, 
	with the following property:
\begin{equation}\label{Carleman-ineqq}
	\begin{array}{l}
		\dis \iint_{Q_T}   \rho_0^{-2} |\pvec|^2  d\xvec\,dt +\|\pvec(\cdot,0)\|^2_{\Vvec'}
		\leq C \left( 
		\iint_{Q_T} \rho_*^{-2} |\Lvec^*\pvec + \nabla\sigma|^2\,d\xvec\,dt 
		+ \iint_{q_T} \rho_0^{-2} |\pvec|^2 \,d\xvec\,dt\right.  										\\ 
		\noalign{\smallskip}
		\dis\phantom{\iint_{Q_T}   \rho_0^{-2} |\pvec|^2  d\xvec\,dt +\|\pvec(\cdot,0)\|^2_{\Vvec'} \leq C}
		\left.+\iint_{Q_T}  |\nabla\cdot \pvec|^2\,d\xvec\,dt\right)
	\end{array}
\end{equation}
	for all $(\pvec,\sigma)\in \Phivec_0$.
\end{proposition}

\begin{proof}
	The proof follows easily by splitting $(\pvec,\sigma)\in \Phivec_0$ in the form
$$
	(\pvec,\sigma)=(\tilde\pvec,\tilde\sigma)+(\hat\pvec,\hat\sigma)
$$
	where $(\hat\pvec,\hat\sigma)$ solves the linear problem
\begin{equation}\label{back_stokes_div_split}
	\left\{
		\begin{array}{lcl}
			\Lvec^*\hat\pvec+\nabla\hat\sigma = \fvec	& \text{in}	&	Q_T,		\\ 
			\noalign{\smallskip}
			\nabla \cdot\hat\pvec=0    					& \text{in}	&	Q_T,   		\\ 
			\noalign{\smallskip}
			\hat\pvec = \0vec						& \text{on}	&	\Sigma_T , \\ 
			\noalign{\smallskip}
			\hat\pvec(\cdot,T) = \pvec(\cdot,T)			& \text{in}	&	\Om
		\end{array}
	\right.
\end{equation}
	with $\fvec:=\Lvec^*\pvec+\nabla\sigma$ and $(\tilde\pvec,\tilde\sigma)$ solves the linear problem
\begin{equation}\label{back_stokes_div}
	\left\{
		\begin{array}{lcl}
			\Lvec^*\tilde\pvec+\nabla\tilde\sigma = \0vec	& \text{in}	&	Q_T,		\\ 
			\noalign{\smallskip}
			\nabla \cdot\tilde\pvec=\nabla \cdot\pvec    			& \text{in}	&	Q_T,   		\\ 
			\noalign{\smallskip}
			\tilde\pvec = \0vec							& \text{on}	&	\Sigma_T,		\\ 
			\noalign{\smallskip}
			\tilde\pvec(\cdot\,,T) = \0vec					& \text{in}	&	\Om.
		\end{array}
	\right.
\end{equation}

	In view of the Carleman estimates \eqref{Carleman-ineq_stokes} for $(\hat\pvec,\hat\sigma)$, we have
\[
	\begin{array}{c}
		\dis \iint_{Q_T} \rho_0^{-2} |\hat\pvec|^2 \,d\xvec\,dt + \| \pvec(\cdot\,,0) \|_{\Vvec}^2
		\leq C \left( \iint_{Q_T} \rho^{-2} |\fvec|^2\,d\xvec\,dt
		+ \iint_{q_T} \rho_0^{-2} |\hat\pvec|^2 \,d\xvec\,dt\right).
	\end{array}
\]	
	 On the other hand, $(\tilde\pvec,\tilde\sigma)$ solves~\eqref{back_stokes_div} in the sense of transposition,  that is, 
\[
	\langle\tilde\pvec,\mat{\psi}\rangle_{\Lvec^2(Q_T),\Lvec^2(Q_T)}
	+\langle\tilde\pvec(\cdot,0),\uvec_0\rangle_{\Vvec',\Vvec}=-\iint_{Q_T} (\nabla\cdot \pvec) \, h \,d\xvec\,dt,
\]
	for all $(\mat{\psi},\uvec_0)\in \Lvec^2(Q_T)\times \Vvec$, where $(\uvec,h)$ is the unique strong solution to
\[
	\left\{
		\begin{array}{lcl}
			\Lvec\uvec + \nabla h = \mat{\psi}	& \text{in}	&	Q_T,		\\ 
			\noalign{\smallskip}
			\nabla \cdot\uvec = \0vec    			& \text{in}	&	Q_T,   		\\ 
			\noalign{\smallskip}
			\uvec = \0vec					& \text{on}	&	\Sigma_T,		\\ 
			\noalign{\smallskip}
			\uvec(\cdot\,,0) = \uvec_0			& \text{in}	&	\Om.
		\end{array}
	\right.
\]
	Consequently, we can argue as in~\cite{RAYMOND} and deduce that
\[
	\|\tilde\pvec\|^2_{\Lvec^2(Q_T)}+\|\tilde\pvec(\cdot,0)\|^2_{\Vvec'}\leq C\,\|\nabla\cdot \pvec\|^2_{L^2(Q_T)}.
\]	
	
	Now, putting together the estimates for $(\hat\pvec,\hat\sigma)$ and $(\tilde\pvec,\tilde\sigma)$,  we are easily led easily to~\eqref{Carleman-ineqq}.
\end{proof}

\

   	Let us introduce the bilinear form
$$
	\tilde{\mvec}((\pvec,\sigma),(\pvec',\sigma')) := \mvec((\pvec,\sigma),(\pvec',\sigma'))
	+\iint_{Q_T} (\nabla\cdot \pvec) (\nabla\cdot \pvec')\,d\xvec\,dt.
$$
	Again, in view of the unique continuation property of the Stokes system, $\tilde{\mvec}(\cdot\,,\cdot)$ is a scalar product 
	in~$\Phivec_0$.
	
	Let $\Phivec$ be the completion of $\Phivec_0$ with respect to this scalar product. As before, $\Phivec$ is a Hilbert space, 
	the functions $(\pvec,\sigma) \in \Phivec$ satisfy
\begin{equation}\label{finite-rhss}
	\iint_{Q_T} \rho^{-2} |\Lvec^*\pvec + \nabla\sigma|^2 \,d\xvec\,dt+ \iint_{q_T} \rho_0^{-2} |\pvec|^2\,d\xvec\,dt
	+\iint_{Q_T} |\nabla\cdot \pvec|^2 d\xvec\,dt < +\infty
\end{equation}
	and, from Proposition~\ref{prop-Carleman-Stokess} and a density argument, we also have \eqref{Carleman-ineqq} 
	for all $(\pvec,\sigma) \in \Phivec$.
   	
	On the other hand, any $(\pvec,\sigma) \in \Phivec$ satisfies
$$
	\rho_0^{-1}\pvec \in \Lvec^2(Q_T),\quad \pvec(\cdot,0)\in \Vvec'
$$
	and
\begin{equation}\label{cont-C00}
	\|\pvec(\cdot\,,0) \|_{\Vvec'}^2 \leq C \, \tilde{\mvec}((\pvec,\sigma),(\pvec,\sigma)) \quad \forall (\pvec,\sigma) \in \Phivec.
\end{equation}
	
   	By setting
\begin{equation}\label{defB00}
	\langle \mat{\tilde{\ell}},(\pvec,\sigma) \rangle := \langle \pvec(\cdot,0),\yvec_0\rangle_{\Vvec',\Vvec}
\end{equation}
	thanks to \eqref{cont-C00}, we have that $ \mat{\tilde{\ell}}$ is continuous on $\Phivec$.

\

	Let us introduce the space
$$
	\tilde M =  \{\, \mu\in L^2(Q_T) : \int_\Om \mu(\xvec,t) \, d\xvec= 0 \,~\hbox{a.e.} \,\}
$$ 
	
	and the following reformulation of~\eqref{pb-psigma-s}:
\begin{equation}\label{pb-psigma-m}
	\left\{
		\begin{array}{l}
			\dis \tilde{\mvec}((\pvec,\sigma),(\pvec',\sigma'))  + \iint_{Q_T}\left(\nabla\cdot\pvec'\right)\mu\,d\xvec\,dt 
			=\langle \tilde{\mat{\ell}},(\pvec,\sigma) \rangle,\\ 
			\noalign{\smallskip}
			\dis \iint_{Q_T}\left( \nabla\cdot \pvec \right) \mu'\,d\xvec\,dt  = 0,\\ 
			\noalign{\smallskip}
			\dis \qquad \forall (\pvec',\sigma',\mu') \in \Phivec\times \tilde M; \ (\pvec,\sigma,\mu) 
			\in \Phivec\times \tilde M.
		\end{array}
	\right.
\end{equation}
   
	Once more, notice that the definitions of $\Phivec$ and~$\tilde M$ are the appropriate to keep all the terms 
	in \eqref{pb-psigma-m} meaningful.

   	Let the bilinear forms $\tilde\avec(\cdot\,,\cdot)$ and $\tilde\bvec(\cdot\,,\cdot)$ be given by
$$
	\tilde\avec((\pvec,\sigma),(\pvec',\sigma')) := \mvec((\pvec,\sigma),(\pvec',\sigma'))
$$
	and 
$$
	\tilde\bvec(\pvec,\sigma,\mu) := \iint_{Q_T}\left( \nabla\cdot \pvec \right) \mu\,d\xvec\,dt.
$$

	Then, $\tilde\avec(\cdot\,,\cdot)$ and $\tilde\bvec(\cdot\,,\cdot)$ are well-defined and continuous and
	\eqref{pb-psigma-m} reads:
\begin{equation}\label{pb-psigma-mm}
	\left\{
		\begin{array}{l}
			\dis \tilde\avec((\pvec,\sigma),(\pvec',\sigma')) + \tilde\bvec(\pvec',\sigma',\mu) 
			= \langle \mat{\tilde{\ell}},(\pvec',\sigma') \rangle,							\\ 
			\noalign{\smallskip}
			\dis \tilde\bvec((\pvec,\sigma),\mu') = 0,									\\ 
			\noalign{\smallskip}
			\dis \qquad \forall (\pvec',\sigma',\mu') \in \Phivec\times\tilde  M; \ 
			(\pvec,\sigma,\mu) \in \Phivec\times\tilde  M.
		\end{array}
	\right.
\end{equation}

	One has the following:
\begin{proposition}\label{prop-equiv}
	There exists exactly one solution to~\eqref{pb-psigma-mm}. Furthermore, \eqref{pb-psigma-s} and~\eqref{pb-psigma-mm} 
	are equivalent problems in the following sense:
\begin{enumerate}
	\item 
		If $(\pvec,\sigma,\mu)$ solves \eqref{pb-psigma-mm}, then $(\pvec,\sigma)$ solves \eqref{pb-psigma-s}.
	\item 
		If $(\pvec,\sigma)$ solves \eqref{pb-psigma-s}, there exists $\mu \in \tilde M$ such that $(\pvec,\sigma,\mu)$
		solves~\eqref{pb-psigma-mm}.
\end{enumerate}
\end{proposition}
\begin{proof}
	Let us set
$$
	\tilde{\Psivecc}:=\{\,(\pvec,\sigma)\in \Phivec: \tilde\bvec(\pvec,\sigma,\mu) = 0
	\quad \forall \mu\in \tilde M\,\}.
$$

	We will check that
\begin{itemize}
	\item 
		$\tilde\avec(\cdot\,,\cdot)$ is coercive in $\tilde{\Psivecc}$.
	\item 
		$\tilde\bvec(\cdot\,,\cdot)$ satisfies the usual ``inf-sup" condition in $\Phivec\times \tilde M$.
\end{itemize}

	The proofs of these assertions are straightforward. Indeed, we have $\tilde{\Psivecc}=\Phivecc$
	(the completion of $\Phivecc_0$ with respect to $\mvec(\cdot,\cdot)$, see \eqref{defm}). Thus,
\[
	\begin{array}{lll}
		\dis\tilde\avec((\pvec,\sigma),(\pvec,\sigma))=\dis \mvec((\pvec,\sigma),(\pvec,\sigma)) = \tilde{\mvec}((\pvec,\sigma),(\pvec,\sigma))\quad \forall (\pvec,\sigma)\in \tilde{\Psivecc}
	\end{array}
\]
	and this proves that $\tilde\avec(\cdot\,,\cdot)$ is coercive in $\tilde{\Psivecc}$. 
	
	On the other hand,
	the ``inf-sup'' condition is a consequence of the fact that, for any $\mu\in\tilde  M$, there exists 
	$(\pvec,\sigma)\in \Phivec$ such that 
\begin{equation}\label{inf_sup_div_tilde}
	\tilde\bvec(\pvec,\sigma,\mu)=\|\mu\|^2_{\tilde M}\quad\hbox{and}
	\quad\|(\pvec,\sigma)\|_{\Phivec}\leq C \|\mu\|_{\tilde M}.
\end{equation}
	This can be seen as follows: for any fixed $\mu\in \tilde M$, let $(\pvec,\sigma)$ be the solution to
\[
	\Lvec^*\pvec+\nabla \sigma=\0vec \,~ \textrm{in}\,~ Q_T,
	\quad 
	\nabla\cdot\pvec=\mu \,~ \textrm{in}\,~ Q_T, 
	\quad 
	\pvec=\0vec\,~\textrm{ on }~\,\Sigma_T, 
	\quad 
	\pvec(\cdot,T)=\0vec \,~ \textrm{in}\,~ \Om;
\]	
	then $\pvec$ belongs to $\Lvec^2(Q_T)$ and  
\begin{equation}
	\Vert \pvec\Vert_{\Lvec^2(Q_T)} \leq  C\Vert \mu\Vert_{L^2(Q_T)}
\end{equation}
	(see \cite{RAYMOND}). Therefore, it is lear that \eqref{inf_sup_div_tilde} also holds.
\end{proof}

%%%%%%%%%%%%%%%%%%%%%%%%%%%%%%%%%%
%%%%  NEW SUBSECTION 
%%%%%%%%%%%%%%%%%%%%%%%%%%%%%%%%%%

\subsection{A third mixed reformulation of \eqref{pb-psigma-mm} with an additional multiplier}\label{Sec-mixed_S3}

	As in Section \ref{Sec-mixed_S1}, introducing the variable
$$
	\zvec:=\Lvec^*\pvec+\nabla \sigma,
$$
	we observe that \eqref{pb-psigma-mm} is equivalent to:
\begin{equation}\label{pb-psigma-mmm_c1}
	\left\{
		\begin{array}{l}
			\dis \hat{\avec}((\zvec,\pvec,\sigma),(\zvec',\pvec',\sigma')) 
			+ \hat{\bvec}((\zvec',\pvec',\sigma'),(\mat{\lambda},\mu)) 
			= \langle \mat{\hat{\ell}},(\zvec',\pvec',\sigma') \rangle,			\\ 
			\noalign{\smallskip}
			\dis  \hat{\bvec}((\zvec,\pvec,\sigma),(\mat{\lambda}',\mu')) = 0,	\\ 
			\noalign{\smallskip}
			\dis \qquad \forall ((\zvec',\pvec',\sigma'),(\mat{\lambda}',\mu')) 
			\in \Zvec \times \tilde\Phivecc \times \Lavec \times \tilde M;
			\ ((\zvec,\pvec,\sigma),(\mat{\lambda},\mu)) 
			\in \Zvec \times \tilde\Phivecc \times \Lavec \times \tilde M,
		\end{array}
	\right.
\end{equation}
	where the bilinear forms $\hat{\avec}(\cdot\,,\cdot)$ and~$\hat{\bvec}(\cdot\,,\cdot)$ are given by
$$
	\hat{\avec}((\zvec,\pvec,\sigma),(\zvec',\pvec',\sigma')) := \iint_{Q_T} \left( \rho^{-2}\zvec\cdot\zvec' 
	+ \rho_0^{-2}\pvec\cdot\pvec'1_\omega  \right)\,d\xvec\,dt
$$
	and
\begin{equation}\label{defhatb}
	\hat{\bvec}((\zvec,\pvec,\sigma),(\mat{\lambda},\mu))
	:=\dis\iint_{Q_T} \left[\zvec-(\Lvec^*\pvec+\nabla\sigma)\right]\cdot\mat{\lambda}\,d\xvec\,dt 
	+  \iint_{Q_T}\left( \nabla\cdot \pvec \right) \mu\,d\xvec\,dt \\
\end{equation}
	and the linear form $\mat{\hat{\ell}}$ is given by
$$
	\langle \mat{\hat{\ell}},(\zvec,\pvec,\sigma) \rangle := \int_\Om  \yvec_0(\xvec) \cdot \pvec(\xvec,0)\,d\xvec.
$$

   	Now, the following holds:
\begin{proposition}\label{prop-equiv-1}
   	There exists exactly one solution to~\eqref{pb-psigma-mmm_c1}.
   	Furthermore, \eqref{pb-psigma-mm} and~\eqref{pb-psigma-mmm_c1} are equivalent problems in the following sense:
\begin{enumerate}
	\item 
		If $((\zvec,\pvec,\sigma),(\mat{\lambda},\mu))$ solves \eqref{pb-psigma-mmm_c1}, then 
		$(\pvec,\sigma,\mu)$ solve \eqref{pb-psigma-mm}.
	\item 
		If $(\pvec,\sigma,\mu)$ solves \eqref{pb-psigma-mm}, there exists $\mat{\lambda} \in \Lavec$ such that 
		$((\zvec,\pvec,\sigma),(\mat{\lambda},\mu))$, with
		$$
			\zvec := \Lvec^*\pvec + \nabla \sigma,
  		$$
		solves~\eqref{pb-psigma-mmm_c1}.
\end{enumerate}
\end{proposition}
\begin{proof}
Let us introduce the space
$$
	\hat{\Psivecc}=\{\,(\zvec,\pvec,\sigma)\in \Zvec \times \tilde\Phivecc :
	\hat\bvec((\zvec,\pvec,\sigma),(\mat{\lambda},\mu)) = 0\quad \forall (\mat{\lambda},\mu)\in \Lavec \times \tilde M \,\}
$$
	and, as before, let us check that
\begin{itemize}
	\item 
		$\hat\avec(\cdot\,,\cdot)$ is coercive in $\hat{\Psivecc}$.
	\item 
		$\hat\bvec(\cdot\,,\cdot)$ satisfies the usual ``inf-sup" condition in $(\Zvec \times \tilde\Phivecc) \times (\Lavec \times \tilde M)$.
\end{itemize}

	Again, the proofs of these assertions are easy. Indeed,  for any 
	$(\zvec,\pvec,\sigma)\in \hat{\Psivecc}$, $\zvec=\Lvec^*\pvec+\nabla \sigma$ and $\nabla\cdot \pvec=0$ and,
	 therefore,
\[
	\begin{alignedat}{2}
		\noalign{\smallskip}\dis
		\hat{\avec}((\zvec,\pvec,\sigma),(\zvec,\pvec,\sigma))	
		=&~\iint_{Q_T} \left( \rho^{-2}|\zvec|^2 + \rho_0^{-2}|\pvec|^2\,1_\om\right)\,d\xvec\,dt	\\
		\noalign{\smallskip}\dis
		=&~{1\over2}\|(\zvec,\pvec,\sigma)\|^2_{\Zvec \times \tilde\Phivecc}
		+{1\over2}\iint_{Q_T}\rho_0^{-2}|\pvec|^21_\om\,d\xvec\,dt							\\
		\noalign{\smallskip}\dis
		\geq&~{1\over2}\|(\zvec,\pvec,\sigma)\|^2_{\Zvec \times \tilde\Phivecc},
	\end{alignedat}
\]
	whence $\hat\avec(\cdot\,,\cdot)$ is coercive in $\hat{\Psivecc}$. 
	
	On the other hand, the ``inf-sup'' condition is a consequence of the fact that, for any $(\mat{\lambda},\mu)\in \Lavec \times \tilde M$,
	 there exists $(\zvec,\pvec,\sigma)\in \Zvec \times \tilde\Phivecc$ such that 
\begin{equation}\label{inf_sup_div}
	\hat{\bvec}((\zvec,\pvec,\sigma),(\mat{\lambda},\mu))=\|(\mat{\lambda},\mu)\|^2_{\Lavec \times \tilde M} \quad\hbox{and} 
	\quad\|(\zvec,\pvec,\sigma)\|_{\Zvec \times \tilde\Phivecc}\leq C \|(\mat{\lambda},\mu)\|_{\Lavec \times \tilde M}.
\end{equation}
	This time, the argument is as follows: for any fixed $(\mat{\lambda},\mu)\in \hat{\Yvec}$, let $(\pvec,\sigma)$ be the solution to
\[
\Lvec^*\pvec+\nabla \sigma=\0vec \,~ \textrm{in}\,~ Q_T,\quad \nabla\cdot\pvec=\mu \,~ \textrm{in}\,~ Q_T, \quad \pvec=0\,~\textrm{ on }~\,\Sigma_T, \quad \pvec(\cdot,T)=\0vec \,~ \textrm{in}\,~ \Omega;
\]
	then $\pvec$ belongs to $\Lvec^2(Q_T)$ and  
\begin{equation}
	\Vert \pvec\Vert_{\Lvec^2(Q_T)} \leq  C\Vert \mu\Vert_{L^2(Q_T)}.
\end{equation}
	Taking $\zvec=\rho^2\mat{\lambda}$, one arrives easily at~\eqref{inf_sup_div}.
\end{proof}

%%%%%%%%%%%%%%%%%%%%%%%%%%%%%%%%%%
%%%%  NEW SUBSECTION 
%%%%%%%%%%%%%%%%%%%%%%%%%%%%%%%%%%

\subsection{Another formulation related to \eqref{reform_pb-psigma-s}}\label{Sec-mixed_S4}

	Let us introduce the spaces
\[
	\begin{alignedat}{2}
		\Phivecc^*:=~&\{\, (\pvec,\sigma) : \,\dis \iint_{Q_T}\left[\rho_2^{-2}|\pvec_t|^2+\rho_1^{-2}|\nabla \pvec|^2
		+\rho_0^{-2}|\pvec|^2+\rho^{-2} |\nabla\sigma|^2\right]d\xvec\,dt< +\infty,							      \\
		&  \,~\nabla\cdot\pvec \equiv 0,~\pvec=\0vec \ \text{ on } \ \Sigma_T \,\},   						      \\
		\noalign{\smallskip}
		\Lavec^*:=~&\{\, \mat{\lambda}:\ \dis\iint_{Q_T}\left[ \rho_2^{2}|\mat{\lambda}|^2 +\rho_1^{2}
		|\nabla \mat{\lambda}|^2\right]d\xvec\,dt< +\infty,~\mat{\lambda} = \0vec \ \text{ on } \ \Sigma_T \,\},
	\end{alignedat}
\]
	the bilinear forms $\avec^*(\cdot\,,\cdot)$ and~$\bvec^*(\cdot\,,\cdot)$, with
$$
	\avec^*((\zvec,\pvec,\sigma),(\zvec',\pvec',\sigma')) := \iint_{Q_T} \left( \rho^{-2}\zvec\cdot\zvec' 
	+ \rho_0^{-2}\pvec\cdot\pvec'1_\omega  \right)\,d\xvec\,dt
$$
	and
\[
	\bvec^*((\zvec,\pvec,\sigma),\mat{\lambda}):=\dis\iint_{Q_T} 
	\left\{\left[\zvec+\pvec_t-\nabla\sigma\right] \cdot \mat\lambda - \nu\nabla \pvec\cdot \nabla\mat\lambda\right\}\,d\xvec\,dt\\
\]
	and the linear form $\mat{\ell^*}$, with
$$
	\langle\mat{\ell^*},(\zvec,\pvec,\sigma) \rangle := \int_\Om  \yvec_0(\xvec) \cdot \pvec(\xvec,0)\,d\xvec.
$$

	The bilinear form $\bvec^*(\cdot,\cdot)$ appears when we integrate by parts the second--order terms in~$\bvec(\cdot,\cdot)$,
	see \eqref{b_mix_s_1}.
   	Accordingly, at least formally, we can reformulate~\eqref{reform_pb-psigma-s} as follows:
\begin{equation}\label{pb-psigma-mmm_ref}
	\left\{
		\begin{array}{l}
			\dis \avec^*((\zvec,\pvec,\sigma),(\zvec',\pvec',\sigma')) + \bvec^*((\zvec',\pvec',\sigma'),\mat{\lambda}') 
			= \langle \mat{\ell^*},(\zvec',\pvec',\sigma') \rangle,	\\ 
			\noalign{\smallskip}
			\dis  \bvec^*((\zvec,\pvec,\sigma),\mat{\lambda}) = 0,															\\ 
			\noalign{\smallskip}
			\dis \qquad \forall ((\zvec',\pvec',\sigma'),\mat{\lambda}') \in \Zvec \times \Phivecc^* \times \Lavec^*; \ 
			((\zvec,\pvec,\sigma),\mat{\lambda}) \in \Zvec \times \Phivecc^* \times \Lavec^*.
		\end{array}
	\right.
\end{equation}

   	This can be viewed as a new mixed formulation of~\eqref{pb-psigma-s}. However, that these two problems are equivalent 
	in the sense of Propositions~\ref{prop-equiv} and~\ref{prop-equiv-1} is, at present, an open question.

%%%%%%%%%%%%%%%%%%%%%%%%%%%%%%%%%%
%%%%  NEW SUBSECTION 
%%%%%%%%%%%%%%%%%%%%%%%%%%%%%%%%%%

\subsection{A fifth (and final) mixed formulation}\label{Sec-mixed_S5}

	Finally, let us introduce the space
\[
	\begin{alignedat}{2}
		\overline{\Phivecc}:=~&\{\, (\pvec,\sigma) :\,\, 
		\dis \iint_{Q_T}\!\!\!\!\left[ \rho_2^{-2}|\pvec_t|^2 \!+\!\rho_1^{-2}|\nabla \pvec|^2 \!+\!\rho_0^{-2}|\pvec|^2\!
		+\!|\nabla\cdot\pvec|^2\!+\! \rho^{-2} |\nabla\sigma|^2\right] 
		\,d\xvec\,dt< +\infty, \\
		\noalign{\smallskip}
		&\pvec=\0vec \ \text{ on } \ \Sigma_T \,\},	
	\end{alignedat}
\]
	the bilinear forms $\overline{\avec}(\cdot\,,\cdot)$ and~$\overline{\bvec}(\cdot\,,\cdot)$, with
$$
	\overline{\avec}((\zvec,\pvec,\sigma),(\zvec',\pvec',\sigma')) := \iint_{Q_T} \left( \rho^{-2}\zvec\cdot\zvec' 
	+ \rho_0^{-2}\pvec\cdot\pvec'1_\omega  \right)\,d\xvec\,dt
$$
	and
\[
	\overline{\bvec}((\zvec,\pvec,\sigma),(\lavec,\mu)):=\dis\iint_{Q_T} \left\{\left[\zvec+\pvec_t-\nabla\sigma\right] 
	\cdot \mat\lambda-\nu\nabla \pvec\cdot \nabla\mat\lambda\right\}\,d\xvec\,dt 
	+\iint_{Q_T}\left( \nabla\cdot \pvec \right) \mu\,d\xvec\,dt\\
\]
	and the linear form $\overline{\mat{\ell}}$, with
$$
	\langle \overline{\mat{\ell}},(\zvec,\pvec,\sigma) \rangle := \int_\Om  \yvec_0(\xvec) \cdot \pvec(\xvec,0)\,d\xvec.
$$

   	In accordance with \eqref{pb-psigma-mmm_ref}, it can be accepted that, at least formally, \eqref{pb-psigma-mmm_c1} 
	possesses the following reformulation:
\begin{equation}\label{pb-psigma-mmm}
	\left\{
		\begin{array}{l}
			\dis \overline{\avec}((\zvec,\pvec,\sigma),(\zvec',\pvec',\sigma')) 
			+ \overline{\bvec}((\zvec',\pvec',\sigma'),(\lavec,\mu)) 
			= \langle \overline{\mat{\ell}},(\zvec',\pvec',\sigma') \rangle,						\\ 
			\noalign{\smallskip}
			\dis  \overline{\bvec}((\zvec,\pvec,\sigma),(\lavec',\mu')) = 0,						\\ 
			\noalign{\smallskip}
			\dis \qquad \forall ((\zvec',\pvec',\sigma'),(\mat{\lambda}',\mu')) \in \Zvec \times \overline{\Phivecc} \times \Lavec \times M; \ 
			((\zvec,\pvec,\sigma),(\mat{\lambda},\mu)) \in \Zvec \times \overline{\Phivecc} \times \Lavec \times M.
		\end{array}
	\right.
\end{equation}
\begin{remark}{\rm
	The previous mixed formulations possess several relevant properties:
	
\begin{itemize}

	\item By constructing finite dimensional subspaces of $\Phivecc$, we are led to standard mixed approximations of \eqref{reform_pb-psigma-s}.
	But this is not a simple task:
	recall that, in order to have $(\pvec,\sigma)\in \Phivecc$, we need (among other things) $\Lvec^*\pvec+\nabla\sigma\in \Lvec^2(\rho^{-1};Q_T)$ and $\nabla\cdot\pvec=0$.
	\item  Contrarily, it is relatively easy to construct numerically efficient finite dimensional subspaces of $\Phivec$, for instance, based on the {\it Bell triangle} or the {\it Bogner-Fox-Schmidt rectangle}.
	Consequently, we can get finite element approximations of~\eqref{pb-psigma-mm} for which, furthermore, a convergence analysis can be performed.
	\item The same can be said for \eqref{pb-psigma-mmm_c1}.
	In this case, the fact that the variable $\zvec$ appears explicitly is useful for a direct computation of an approximation of the state.
	\item The mixed formulations \eqref{pb-psigma-mmm_ref} and \eqref{pb-psigma-mmm} share an advantageous characteristic: 
	they can be approximated in a rather standard way by $C^0$ finite elements since, after integration by parts, no second-order spatial derivative appears.
	Unfortunately, up to our knowledge, it is unknown whether or not they are well posed.
	More precisely, the proof of the ``inf-sup'' condition is open and, moreover, the well-posedness of their associated discrete versions is not clear. 
\end{itemize}}
\end{remark}

%%%%%%%%%%%%%%%%%%%%%%%%%%%%%%%%%%
%%%%  NEW SUBSECTION 
%%%%%%%%%%%%%%%%%%%%%%%%%%%%%%%%%%

\subsection{A numerical approximation of \eqref{pb-psigma-mmm} (without justification)}\label{Sec-mixed_S6}

	Let us conserve the notation in~Section~\ref{mixed_heat}.

	For any couple of integers $m, n\geq 1$, we will set
\[
		\begin{array}{l}
			\noalign{\smallskip}\dis
			\overline\Zvec_h(m,n):= \{\, \zvec_h \in \Cvec^0(\overline{Q}_{\kappa,T}) : \zvec_h|_K \in 
			(\mathbb{P}_{m,\xvec} \otimes \mathbb{P}_{n,t})(K) \ \ \forall K \in  {\mathcal Q}_{h} \,\},	\\     
			\noalign{\smallskip}\dis
			\overline\Vvec_h(m,n) := \{\, \pvec_h \in \overline\Zvec_h(m,n) : \pvec_h = \0vec \ \hbox{ on } \ \Sigma_T \,\},	\\
			\noalign{\smallskip}\dis
			\overline M_h(m,n):= \{\, \sigma_h \in C^0(\overline{Q}_{\kappa,T}) : \sigma_h|_K \in  (\mathbb{P}_{m,\xvec} \otimes 
			\mathbb{P}_{n,t})(K) \ \ \forall K \in  {\mathcal Q}_{h} \,\}.
		\end{array}
\]

	Then, $\overline\Zvec_h(m,n)$ and $\overline\Vvec_h(m,n)$ are finite dimensional subspaces of the Hilbert space 
	$\Hvec^1(Q_{\kappa,T})$. Moreover, $\overline\Vvec_h(m,n)\times \overline M_h(m,n)\subset \overline{\Phivecc}$, 
	$\overline\Vvec_h(m,n) \subset \overline{\Lavec}$ and  $\overline M_h(m,n) \subset \tilde{M}$. Therefore, for 
	any $m$, $n$, $m'$, $n'$, $m''$, $n''$, $m'''$, $n'''$, $m''''$, $n'''' \geq 1$, we can define 
$$
	\overline\Xvec_h := \overline\Zvec_h(m,n) \times \overline\Vvec_h(m',n')\times\overline M_h(m'',n'')
	\hbox{ and } \overline\Yvec_h := \overline\Vvec_h(m''',n''') \times\overline M_h(m'''',n''''),
$$ 
	that are finite dimensional subspaces of $\Zvec \times \overline{\Phivecc}$ and $\Lavec \times M$, respectively. 

   	The following mixed approximation of~\eqref{pb-psigma-mmm} makes sense:
\begin{equation}\label{pb-psigma-mm-h}
	\left\{
		\begin{array}{l}
			\dis \overline{\avec}((\zvec_h,(\pvec_h,\sigma_h)),(\zvec'_h,(\pvec'_h,\sigma'_h))) 
			+ \overline{\bvec}((\zvec'_h,(\pvec'_h,\sigma'_h)),(\lavec_h,\mu_h)) 
			= \langle \overline{\mat{\ell}},(\zvec'_h,(\pvec'_h,\sigma'_h)) \rangle,				\\ 
			\noalign{\smallskip}
			\dis  \overline{\bvec}((\zvec_h,(\pvec_h,\sigma_h)),(\lavec'_h,\mu'_h)) = 0,			\\ 
			\noalign{\smallskip}
			\dis \qquad \forall ((\zvec'_h,(\pvec'_h,\sigma'_h)),(\mat{\lambda}'_h,\mu'_h)) 
			\in \overline{\Xvec}_h\times \overline{\Yvec}_h; \ ((\zvec_h,(\pvec_h,\sigma_h)),
			(\mat{\lambda}_h,\mu_h)) \in \overline{\Xvec}_h\times\overline{\Yvec}_h.
		\end{array}
	\right.
\end{equation}

%%%%%%%%%%%%%%%%%%%%%%%%%%%%%%%%%%
%%%%  NEW SUBSECTION 
%%%%%%%%%%%%%%%%%%%%%%%%%%%%%%%%%%

\subsection{A numerical experiment}\label{Num-S-2}

	This Section deals with some numerical results. We have solved \eqref{pb-psigma-mm-h} with the following data: % AQUI
	$\Om=(0,1)\times(0,1)$, $\omega=(0.2,0.6)\times(0.2,0.6)$, $T=1$, $K_1=1$, $K_2=2$, $\betamod_0=\betamod_{0}^{(0.5,0.5)}$ 
	(as in Section \ref{sec_heat});
	$\nu=1$,  $\yvec_0(\xvec) \equiv (M,0)$, with $M  = 1000$.

%	$\nu=1$,  $\yvec_0(\xvec) \equiv M \nabla\times \psi(\xvec)$, $\psi(x_1,x_2) \equiv x_1^2 x_2^2  [(1-x_1)(1-x_2)]^2$ with~$M=0.1$.

	Again, the computations have been performed with the software {\it Freefem++}, using $P_2$-Lagrange approximations in~$(\xvec,t)$ for all the variables.
	Th domain and the mesh are depicted in~Fig.~\ref{mesh_heat_lin}.
	The linear system in~\eqref{pb-psigma-mm-h} has been solved with the Arrow-Hurwicz algorithm, where we have taken $r=0.01$ and~$s=0.1$.
	The convergence of this algorithm is illustrated in~Table~\ref{Tab-AH-S}, where the first and the second relative errors are given by	
$$
	{ \|(\zvec^{(k+1)}_h,\pvec^{(k+1)}_h,\sigma^{(k+1)}_h) - (\zvec^{(k)}_h,\pvec^{(k)}_h,\sigma^{(k)}_h)\|_{\Lvec^2(Q_T)} 
	\over \|(\zvec^{(k+1)}_h,\pvec^{(k+1)}_h,\sigma^{(k+1)}_h)\|_{\Lvec^2(Q_T)} }
$$
	and
$$
	{ \| (\lavec^{(k+1)}_h,\mu^{(k+1)}_h) - (\lavec^{(k)}_h,\mu^{(k)}_h) \|_{\Lvec^2(Q_T)} \over \| (\lavec^{(k+1)}_h,\mu^{(k+1)}_h) \|_{\Lvec^2(Q_T)} }\,.
$$

	The computed control and state are displayed in Fig.~\ref{control_Stokes_1}--\ref{control_Stokes_4}.
	
%%%% FIGURE %%%

%\vskip-2cm
\begin{figure}[http]
\begin{center}
\begin{minipage}{0.49\textwidth}
\includegraphics[width=\textwidth]{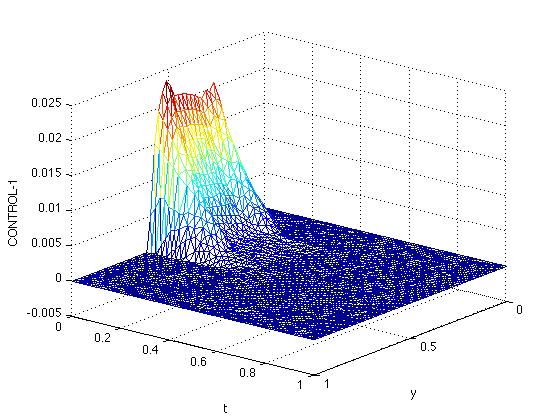}
%\vskip-4cm
\includegraphics[width=\textwidth]{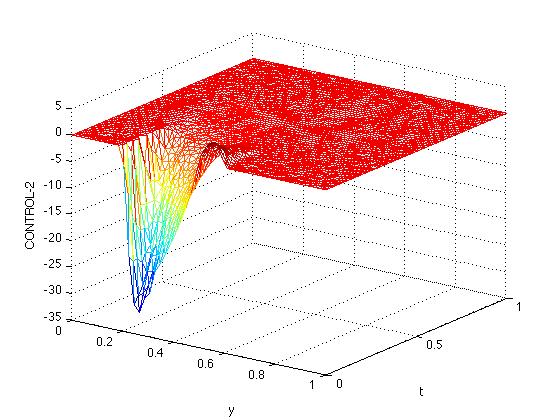}
\end{minipage}
\begin{minipage}{0.49\textwidth}
\includegraphics[width=\textwidth]{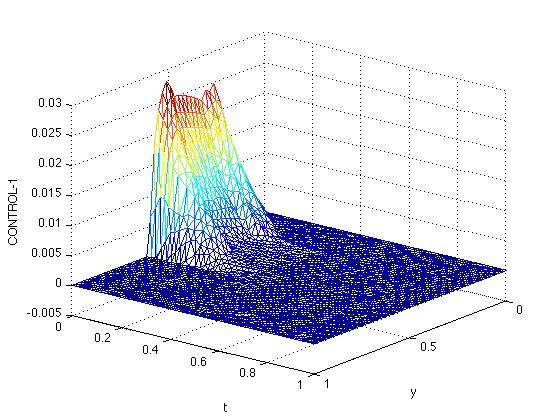}
%\vskip-4cm
\includegraphics[width=\textwidth]{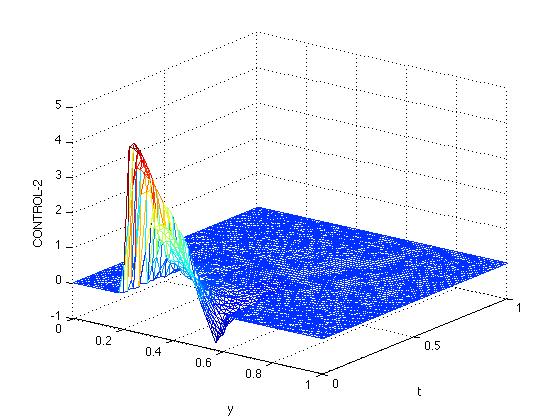}
\end{minipage}
\caption{$\omega=(0.2,0.6)$; $\yvec_0(\xvec)=(1000,0)$. Cuts of $v_{1,h}$  and $10^{10} v_{2,h}$ at $x_1=0.28$  ({\bf Left}) and~$x_1=0.52$ ({\bf Right}).}
\label{control_Stokes_1}
\end{center}
\end{figure}

%%%% FIGURE %%%

%\vskip-2cm
\begin{figure}[http]
\begin{center}
\begin{minipage}{0.49\textwidth}
\includegraphics[width=\textwidth]{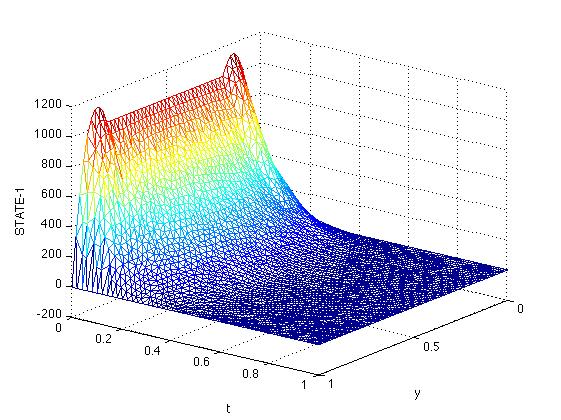}
%\vskip-4cm
\includegraphics[width=\textwidth]{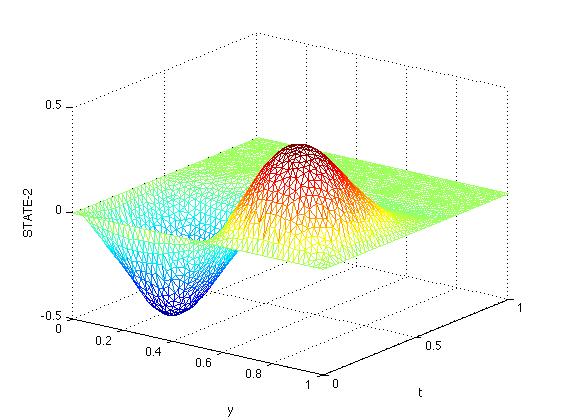}
\end{minipage}
\begin{minipage}{0.49\textwidth}
\includegraphics[width=\textwidth]{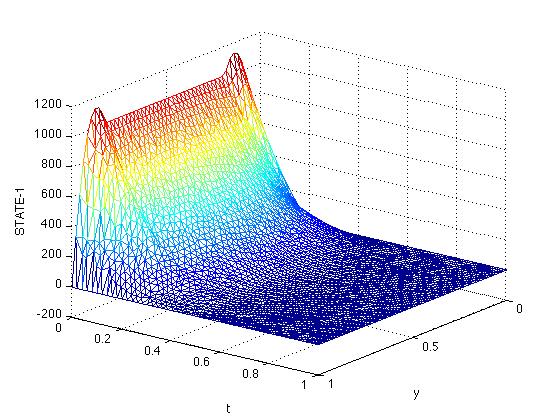}
%\vskip-4cm
\includegraphics[width=\textwidth]{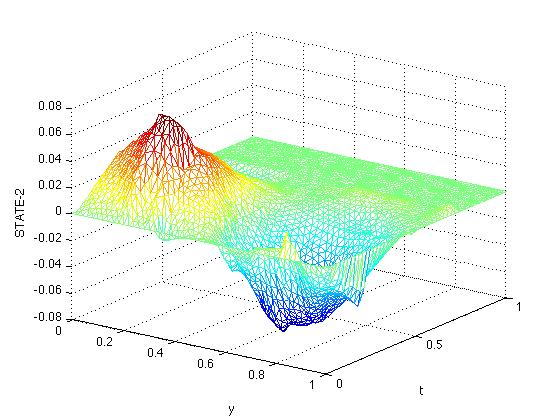}
\end{minipage}
\caption{$\omega=(0.2,0.6)$; $\yvec_0(\xvec)=(1000,0)$. Cuts of $y_{1,h}$  and $10^{10} y_{2,h}$ at $x_1=0.28$  ({\bf Left}) and~$x_1=0.52$ ({\bf Right}).}
\label{control_Stokes_2}
\end{center}
\end{figure}

%%%% FIGURE %%%

\begin{figure}[http]
\begin{center}
\begin{minipage}{0.49\textwidth}
\includegraphics[width=\textwidth]{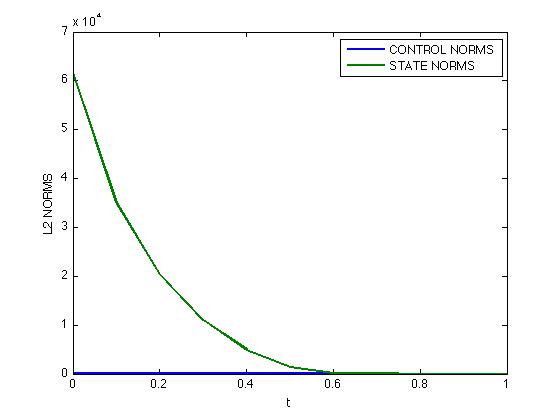}
\end{minipage}
\begin{minipage}{0.49\textwidth}
\includegraphics[width=\textwidth]{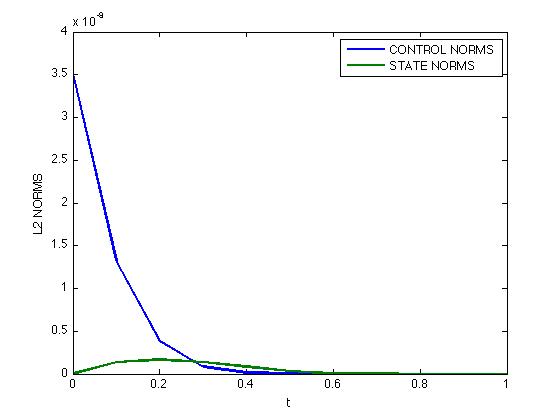}
\end{minipage}
\caption{Evolution of the $L^2$ norms of the first and second components of the control and the state.}
\label{control_Stokes_3}
\end{center}
\end{figure}

%%%% FIGURE %%%

%\vskip-2cm
\begin{figure}[http]
\begin{center}
\begin{minipage}{0.49\textwidth}
\includegraphics[width=\textwidth]{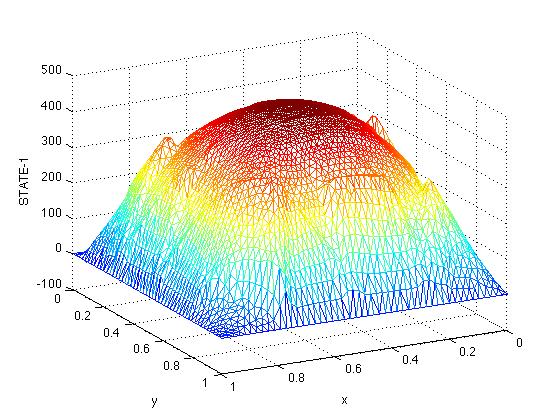}
%\vskip-4cm
\includegraphics[width=\textwidth]{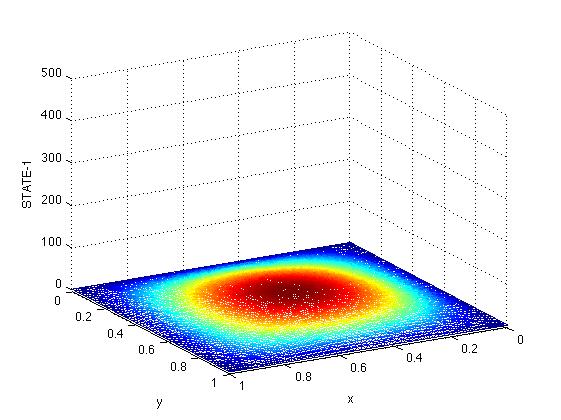}
%\vskip-4cm
\includegraphics[width=\textwidth]{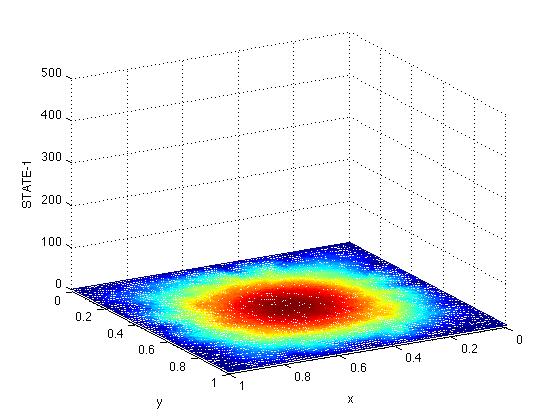}
\end{minipage}
\begin{minipage}{0.49\textwidth}
\includegraphics[width=\textwidth]{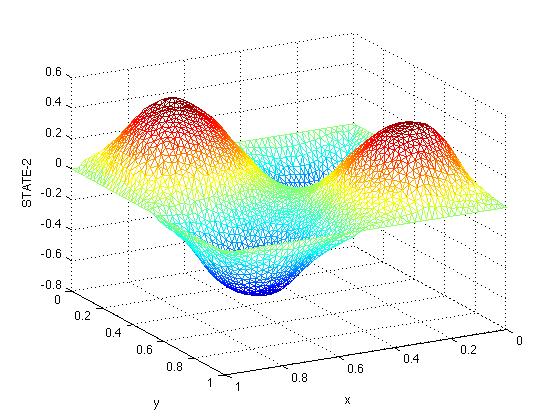}
%\vskip-4cm
\includegraphics[width=\textwidth]{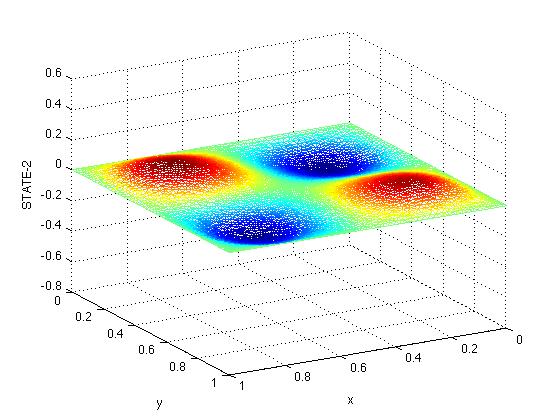}
%\vskip-4cm
\includegraphics[width=\textwidth]{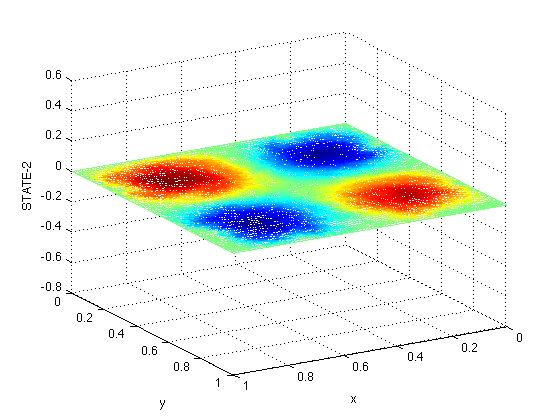}
\end{minipage}
\caption{Evolution of the state at times $t=0.2$, $t=0.5$ and $t=0.8$. Cuts of $y_{1,h}$ ({\bf Left}) and~$10^{10} y_{2,h}$ ({\bf Right}).}
\label{control_Stokes_4}
\end{center}
\end{figure}

%%%% TABLE %%%

\begin{table}[http]
\centering
\begin{tabular}{|c|c|c|}
\hline
Iterate  & Rel.\ error 1 & Rel.\ error 2 \tabularnewline
\hline 
1             & $0.659686$      &  $0.202439$  \tabularnewline\hline 
10           & $0.063864$      &  $0.106203$  \tabularnewline\hline 
20           & $0.016147$      &  $0.076083$  \tabularnewline\hline 
30           & $0.008874$      &  $0.046212$  \tabularnewline\hline 
40           & $0.000464$      &  $0.001397$  \tabularnewline\hline 
50           & $0.000206$      &  $0.000762$  \tabularnewline\hline 
\hline
\end{tabular}
\caption{The behavior of the Arrow-Hurwicz algorithm for~\eqref{pb-psigma-mm-h}.}\label{Tab-AH-S}
\end{table}

%%%% ATTENTION: THE $L^2$ NORMS OF THE CONTROL AND THE STATE VS TIME?
%%% {\color{red} \bf ATTENTION: THE $L^2$ NORMS OF THE CONTROL AND THE STATE VS TIME?}

%%%%%%%%%%%%%%%%%%%%%%%%%%%%%%%%%%
%%%%  NEW SECTION 
%%%%%%%%%%%%%%%%%%%%%%%%%%%%%%%%%%

\section{An application\,: numerical local exact controllability to the trajectories of the Navier-Stokes equations}
\label{Sec_Navier_Stokes}

   	In this Section, we will present a numerical method for the computation of a solution to the local exact controllability problem to the trajectories of \eqref{N-stokes}.
	This controllability property was proved in~\cite{FC-G-P} under suitable regularity assumptions on the trajectories.
	More precisely, we have to assume that the trajectory satisfies
\begin{equation}\label{trajec_hyp}
	\overline{\yvec} \in L^2(0,T;D(\Avec))\cap C^0([0,T];\Vvec)\cap\Lvec^{\infty}(Q_T),
	\quad \overline{\yvec}_t\in L^2\bigl(0,T;\Hvec\bigl),
\end{equation}
	where $D(\Avec):=\Hvec^2(\Om)\cap \Vvec$ is the domain of the usual Stokes operator $\Avec$; see also~\cite{IMANU} for a previous result.

%%%%%%%%%%%%%%%%%%%%%%%%%%%%%%%%%%
%%%%  NEW SUBSECTION 
%%%%%%%%%%%%%%%%%%%%%%%%%%%%%%%%%%

\subsection{A fixed-point algorithm and a mixed formulation}
	
	First of all, let us rewrite the local exact controllability to the trajectories as a null controllability problem.
	To do this, let us put $\yvec=\overline\yvec + \uvec$ and $\pi = \overline \pi + q$ and let us use \eqref{N-stokes}.
	Taking into account that $(\overline \yvec, \overline \pi)$ solves~\eqref{TRAJEC}, we find:
\begin{equation}\label{N-S-REFOR}
	\left\{
		\begin{array}{lcl}
     			\uvec_t-\nu\Delta \uvec+(\uvec\cdot\nabla)\overline\yvec+((\overline\yvec+\uvec)\cdot\nabla)\uvec
			+\nabla q=\vvec1_\om						&\hbox{in}&	Q_T,	\\
    			\nabla \cdot \uvec = 0      						&\hbox{in}&	Q_T, 	\\
     			\uvec = \0vec 						         		&\hbox{on}&	\Sigma_T, 	\\
     			\uvec(0) = \uvec_0 :=\yvec_0-\overline{\yvec}_0               	&\hbox{in}&	\Om.
		\end{array}
	\right.
\end{equation}

	This way, we have reduced our problem to a local null controllability result for the solution $(\uvec,q)$ to the nonlinear 
	problem \eqref{N-S-REFOR}.
	
	Let us suppose that $\uvec_0\in D(\Avec^{\sigma})$, with $1/2<\sigma<1$ ($\Avec^{\sigma}$ is the fractional
	power of the Stokes operator) and let us introduce the fixed-point mapping $F: \Wvec\mapsto \Wvec$, where 
$$
	\Wvec:=\left\{\, \uvec\in \Lvec^\infty(Q_T) : \nabla \cdot \uvec = 0 \hbox{ in } Q_T,
	\ \uvec\cdot\nvec=0\hbox{ on }\Sigma_T\,\right\}.
$$
	Here, for any $\wvec\in \Wvec$,~$\uvec=F(\wvec)$ is, together with some $\vvec$ and $q$, the unique 
	solution to the extremal problem
\begin{equation}\label{variational_F-I-O}
			\left\{
			\begin{array}{l}
			\dis \hbox{Minimize } \ J(\wvec;\uvec,\vvec)=  
			{1\over2}\iint_{Q_T}  \rho^2|\uvec|^2\,d\xvec\,dt + {1\over2}\iint_{q_T}\rho_0^2|\vvec|^2 \,d\xvec\,dt
			\\ \noalign{\smallskip}
			\dis \hbox{Subject to } \ \vvec \in L^2(q_T), \ (\uvec,q,\vvec) \ \text{satisfies} \ \eqref{N-S-REFOR-LL},
			\end{array}
			\right.
\end{equation}
	where \eqref{N-S-REFOR-LL} reads as follows:
\begin{equation}\label{N-S-REFOR-LL}
	\left\{
		\begin{array}{lcl}
     			\uvec_t-\nu\Delta \uvec+(\uvec\cdot\nabla)\overline\yvec+((\overline\yvec+\wvec)\cdot\nabla)\uvec
			+\nabla q=\vvec1_\om  				&\hbox{in}&	Q_T,	\\
    			\nabla \cdot \uvec = 0                     			&\hbox{in}&	Q_T,   	\\
     			\uvec = \0vec 					     	&\hbox{on}&	\Sigma_T,  	\\
     			\uvec(0) = \uvec_0                                     		&\hbox{in}& 	\Om.
		\end{array}
	\right.
\end{equation}
   	It is again assumed that the weights $\rho$ and $\rho_0$ satisfy~\eqref{weights-0}.
		
	We have:
\begin{theorem}\label{th-2_o}
   	For any $\uvec_0 \in D(\Avec^{\sigma})$ and $T > 0$, there exists exactly one solution 
	to~\eqref{variational_F-I-O}-\eqref{N-S-REFOR-LL}.
\end{theorem}
	
	This can be regarded as a consequence of the following {\it Carleman inequality} for Oseen systems
	(the proof can be found in~\cite{Ima-JPP-Yamamoto}):

\begin{proposition}\label{prop-Carleman-Oseen} % AQUI
	For all $R>0$, the function $\betamod_0$ and the associated weights  $\rho$, $\rho_0$ and $\rho_1$ furnished 
	by~Proposition~\ref{prop-Carleman_h} can be chosen such that, for some $C$, only depending on $\Om$, $\om$, 
	$T$ and $R$, and for all $\wvec\in \Wvec$ with $\|\wvec\|_{\Lvec^\infty(Q_T)}\leq R$, one has:
\begin{equation}\label{Carleman-ineq_n-s}
	\begin{array}{c}
		\dis \iint_{Q_T} \!\!\!\!\!\left( \rho_1^{-2} |\nabla\pvec|^2 
		\!+\! \rho_0^{-2} |\pvec|^2 \!+\! \rho^{-2} |\nabla\sigma|^2 \right)d\xvec\,dt 
		\!\leq\! C \!\left( \iint_{Q_T}\!\!\!\!\! \rho^{-2} |\Mvec^*\pvec
 		\!+\! \nabla\sigma|^2\,d\xvec\,dt \!+\!\! \iint_{q_T}\!\!\!\! \rho_0^{-2} |\pvec|^2 \,d\xvec\,dt \right)
	\end{array}
\end{equation}
	for all $(\pvec,\sigma) \in \Phivecc_0$.
	Here, we have used the notation
\[
	\quad \Mvec^*\pvec = -\pvec_t -\nu\Delta \pvec -\nabla\pvec\,(\overline\yvec+\wvec) - \nabla\pvec^t\,\overline\yvec,
	\quad \Mvec\uvec = \uvec_t -\nu\Delta \uvec+(\uvec\cdot\nabla)\overline\yvec+((\overline\yvec+\wvec)\cdot\nabla)\uvec.
\]
\end{proposition}
   		
   	For any $\wvec\in \Wvec$, we will denote by $\mvec(\wvec;\,\cdot\,,\,\cdot\,)$ the following associated bilinear form 
	on~$\Phivecc_0$:
$$
	\mvec(\wvec;(\pvec,\sigma),(\pvec',\sigma')) := \iint_{Q_T}\left(\rho^{-2}(\Mvec^*\pvec+\nabla\sigma)
	\cdot(\Mvec^*\pvec'+\nabla\sigma') +1_\om \rho_0^{-2}\pvec\cdot\pvec' \right)\,d\xvec\,dt ;
$$
	recall that $\Phivecc_0$ is given in~\eqref{defPhi0}.
	
	This bilinear form is a scalar product in $\Phivecc_0$. Let us denote by $\Phivecc^{\wvec}$ the  corresponding completion.
	Then, for a good choice of $\rho$ and $\rho_0$ (the same as above), the solution to~\eqref{variational_F-I-O} can be
	characterized by the identities
\begin{equation}\label{optimal-yv}
	\uvec =  \rho^{-2} (\Mvec^*\pvec_\wvec + \nabla\sigma_\wvec ), \quad \vvec = -\rho_0^{-2}\bigl.\pvec_\wvec\bigr|_{q_T},
\end{equation}
	where $(\pvec_\wvec ,\sigma_\wvec )$ is the solution to a variational equality in the Hilbert space $\Phivecc^{\wvec}$:
\begin{equation}\label{pb-psigma_NS}
	\left\{
		\begin{array}{l}
			\mvec(\wvec;(\pvec_\wvec ,\sigma_\wvec ),(\pvec',\sigma'))=\dis\!\!\int_\Omega \uvec_0(\xvec)
			\cdot\pvec'(\xvec,0)\,d\xvec 														\\
			\noalign{\smallskip}
			\dis \forall (\pvec',\sigma') \in \Phivecc^{\wvec}; \ (\pvec_\wvec ,\sigma_\wvec ) \in \Phivecc^{\wvec}.
		\end{array}
	\right.
\end{equation}

\begin{remark}\label{N-S_remark}{\rm
	Note that, in view of \eqref{Carleman-ineq_n-s}, for any fixed $R>0$, the weights indicated in~Proposition~\eqref{prop-Carleman-Oseen} lead to a family of norms $\mvec(\wvec;\cdot,\cdot)^{1/2}$ that are {\it equivalent} as long as $\|\wvec\|_{\Lvec^\infty(Q_T)}\leq R$.
	Consequently, the associated spaces $\Phivecc^{\wvec}$ are the same for all $\wvec$ with $\|\wvec\|_{\Lvec^\infty(Q_T)}\leq R$.\Fin
	}
\end{remark}

	In order to solve the null controllability problem for \eqref{N-S-REFOR}, it suffices to find a solution to the fixed-point equation
\begin{equation}\label{eq-fp}
	\uvec=F(\uvec),\quad\uvec\in  \Wvec.
\end{equation}

	Moreover, in view of the results in \cite{GBGP}, if $\uvec_0$ is small enough, $F$ is well defined and possesses at least one fixed-point.

	Consequently, a natural strategy is to use the following algorithm:

\vspace{0.3cm}
\noindent\textbf{ALG~2 (Fixed-point):}

\begin{enumerate}

	\item[(i)] 
		Choose $\uvec^0\in \Wvec$.
		
	\item [(ii)]
		Then, for given~$n \!\geq\! 0$~and~$\uvec^n \!\in\! \Wvec$, compute~$\uvec^{n\!+\!1}\!=\!F(\uvec^n)$, i.e.~find the unique solution~$(\uvec^{n+1},\vvec^{n+1})$ to the extremal problem
		\begin{equation}\label{variational_F-I-O-n}
			\left\{
			\begin{array}{l}
			\dis \hbox{Minimize } \ J(\uvec^n;\uvec^{n+1},\vvec^{n+1})=  
			\rho^2|\uvec^{n+1}|^2\,d\xvec\,dt + {1\over2}\iint_{q_T}\rho_0^2|\vvec^{n+1}|^2 \,d\xvec\,dt
			\\ \noalign{\smallskip}
			\dis \hbox{Subject to } \ \vvec^{n+1} \in L^2(q_T), \ (\uvec^{n+1},q^{n+1},\vvec^{n+1}) 
			\ \text{satisfies} \ \eqref{N-S-REFOR-L-n},		
			\end{array}
			\right.
		\end{equation}
where \eqref{N-S-REFOR-L-n} is the following
		\begin{equation}\label{N-S-REFOR-L-n}
			\left\{
				\begin{array}{lll}
     					\uvec^{n+1}_t-\nu\Delta \uvec^{n+1}+(\uvec^{n+1}\cdot\nabla)\overline\yvec
					+((\overline\yvec+\uvec^{n})\cdot\nabla)\uvec^{n+1}+\nabla q^{n+1}=\vvec^{n+1}1_\om  							&\hbox{in}& Q_T,           			\\
    					\nabla \cdot \uvec^{n+1} = 0      					  				         						                     		&\hbox{in}& Q_T, 	          		\\
     					\uvec^{n+1} = \0vec 						          			                 						                     			&\hbox{on}& \Sigma_T,  			\\
     					\uvec^{n+1}(0) = \uvec_0                                        				 		                                                                      					&\hbox{in}& \Om.
				\end{array}
			\right.
		\end{equation}

\end{enumerate}

	This is a classical fixed-point method for~\eqref{eq-fp}.
	We start from a prescribed state $\uvec^0$ and, then, we  solve a null controllability problem for a linear parabolic system at each step.
	This way, we produce a sequence  $\{\uvec^n,\vvec^n\}$ that is expected to converge to a solution to the null controllability problem \eqref{N-S-REFOR}.

	For the numerical solution of the problems \eqref{variational_F-I-O-n}--\eqref{N-S-REFOR-L-n}, we can apply arguments 
	similar to those in Sections~\ref{Sec-mixed_S5} and 
	\ref{Sec-mixed_S6}. Thus, a suitable mixed formulation is:
\begin{equation}\label{pb-psigma-mmm_NS}
	\left\{
		\begin{array}{l}
			\dis \overline{\avec}((\zvec,\pvec,\sigma),(\zvec',\pvec',\sigma')) + \overline{\bvec}((\zvec',\pvec',\sigma'),(\lavec,\mu)) 
			= \langle \overline{\mat{\ell}},(\zvec',\pvec',\sigma') \rangle,	\\ 
			\noalign{\smallskip}
			\dis  \overline{\bvec}((\zvec,\pvec,\sigma),(\lavec',\mu')) = 0,															\\ 
			\noalign{\smallskip}
			\dis \qquad \forall ((\zvec',\pvec',\sigma'),(\mat{\lambda}',\mu')) 
			\in \Zvec \times \overline{\Phivecc} \times \Lambda \times M; \ 
			((\zvec,\pvec,\sigma),(\mat{\lambda},\mu)) \in \Zvec \times \overline{\Phivecc} \times \Lambda \times M.
		\end{array}
	\right.
\end{equation}
where, the spaces $\Zvec$, $\overline{\Phivecc}$, $\Lambda$ and~$M$ and the forms $\overline{\avec}(\cdot\,,\cdot)$, $\overline{\bvec}(\cdot\,,\cdot)$ and~$\overline{\mat{\ell}}$ are defined in~Section~\ref{sec_stokes}.

%%%%%%%%%%%%%%%%%%%%%%%%%%%%%%%%%%
%%%%  NEW SUBSECTION 
%%%%%%%%%%%%%%%%%%%%%%%%%%%%%%%%%%

\subsection{Numerical experiments}\label{Num-S}

	In this Section, we are going to present some numerical experiments concerning the Poiseuille flow $\overline\yvec_{P}$ 
	and the Taylor-Green vortex  $\overline\yvec_{TG}$. In both cases, we try to solve a local exact controllability problem: 
$$
	\overline\yvec(\xvec,T)\equiv\overline\yvec_P(\xvec)\quad\hbox{or}\quad\overline\yvec(\xvec,T)
	\equiv\overline\yvec_{TG}(\xvec,T).
$$
	
	In the case of the Poiseuille flow, we will take the following data: % AQUI
	$\Om=(0,5)\times(0,1)$, $\omega=(1,2)\times(0,1)$, $T=2$, $K_1=1$, $K_2=2$, $\betamod_0=\betamod_0^{(1.5,0.5)}$, $\nu=1$, $\overline{\yvec}_P(x_1,x_2):= (4x_2(1-x_2),0)$, $\yvec_0(\xvec):\equiv\overline\yvec_{p}+M(\nabla\times \psi)(\xvec)$ where $\psi(x_1,x_2)\equiv(x_1x_2)^2[(1-x_1)(1-x_2)]^2$ and $M=0.1$.
	Again, the computations have been performed with the software {\it Freefem++}, using $P_2$-Lagrange approximations and the linear systems have been solved with the Arrow-Hurwicz algorithm, with parameters $r=0.01$ and $s=0.1$.
	
	In the case of the Taylor--Green flow, we have taken the same data, except the following:
	$\Om=(0,\pi)\times(0,\pi)$, $\omega=(\pi/3,2\pi/3)\times(\pi/3,2\pi/3)$, $\psi(x_1,x_2)\equiv(x_1x_2)^2[(\pi-x_1)(\pi-x_2)]^2$, $T=1$ and
$$
	\overline{\yvec}_{TG}(x_1,x_2,t):= (\sin(2x_1)\cos(2x_2)e^{-8 t}, - \cos(2x_1)\sin(2x_2)e^{-8 t}).
$$ 
	The same software and the same kind of approximation were considered.
	
	The computational domains and the corresponding triangulations are displayed in Fig.~\ref{mesh_poiseu} and \ref{mesh_taylor}.
	The behavior of the fixed-point iterates is depicted in Table~\ref{Tab-FP}.
	There, the relative error is given by
$$
	{ \| \uvec^{n+1} - \uvec^n \|_{\Lvec^2(Q_T)} \over  \| \uvec^{n+1} \|_{\Lvec^2(Q_T)} } \,.
$$
	The computed controls and states are shown in Fig.~\ref{fig_poiseu_1} and~\ref{fig_poiseu_2} for the Poiseuille test and
	Fig.~\ref{fig_taylor_1}--\ref{fig_taylor_4} for the Taylor-Green test.
	
	Finally, we have tried to clarify the role of the computed null controls and give an idea of their effect.
	Thus,  in~Fig.~\ref{snormP}, we compare the evolution in time of the $L^2$-norms of $\yvec-\overline{\yvec}_{P}$ and $\zvec-\overline{\yvec}_{TG}$, where $\zvec$ is, together with some pressure, the solution to~\eqref{N-stokes}  with $\vvec = 0$ satisfying the same initial conditions.
	 A similar comparison is furnished in~Fig.~\ref{snormTG} for the Taylor-Green flow test.

%%%% TABLE %%%
 
\begin{table}[http]
\centering
\begin{tabular}{|c|c|c|}
\hline
Iterate  & Rel.\ error (P) & Rel.\ error (TG) \tabularnewline
\hline 
1             & $0.499140$      &  $0.622740$          \tabularnewline\hline 
10           & $0.039318$      &  $0.044985$     \tabularnewline\hline 
20           & $0.010562$      &  $0.012376$   \tabularnewline\hline 
30           & $0.003035$      &  $0.003366$   \tabularnewline\hline 
40           & $0.000331$      &  $0.000851$  \tabularnewline\hline 
50           & $0.000122$      &  $0.000209$  \tabularnewline\hline 
\hline
\end{tabular}
\caption{The behavior of ALG~2 (P: Poiseuille, TG: Taylor-Green).}\label{Tab-FP}
\end{table}
	
%%%% FIGURE %%%

\begin{figure}[http]
\begin{center}
\begin{minipage}{0.6\textwidth}
\includegraphics[width=\textwidth]{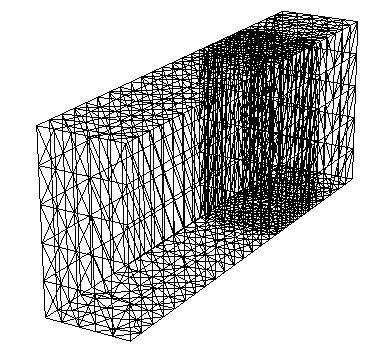}
\end{minipage}
\caption{Poiseuille test -- The domain and the mesh. Number of vertices: 1830. Number of elements (tetrahedra): 7830.
Total number of variables: 12810.}
\label{mesh_poiseu}
\end{center}
\end{figure}

%%%% FIGURE %%%

\begin{figure}[http]
\begin{center}
\begin{minipage}{0.49\textwidth}
\includegraphics[width=\textwidth]{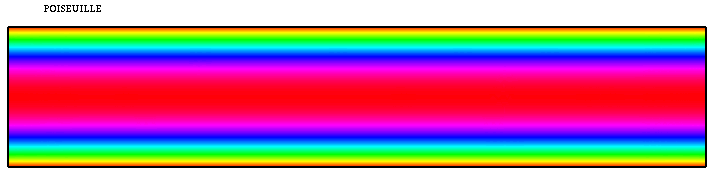}
\includegraphics[width=\textwidth]{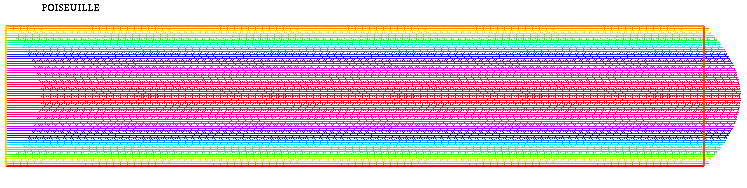}
\end{minipage}
\begin{minipage}{0.49\textwidth}
%\hspace*{0.5cm}
\includegraphics[width=\textwidth]{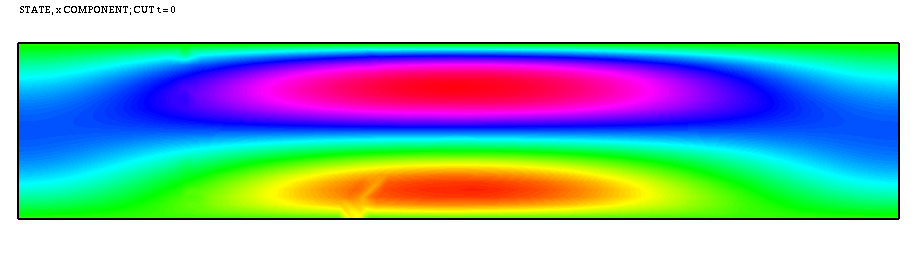}
\includegraphics[width=\textwidth]{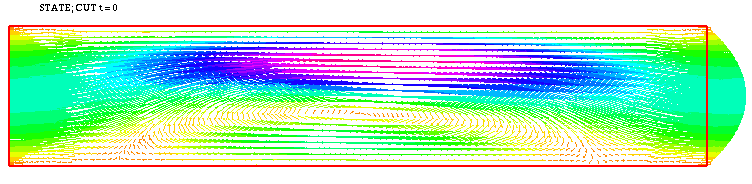}
\end{minipage}
\caption{Poiseuille test -- the target ({\bf Left}) and 
the initial state ({\bf Right}).}
\label{fig_poiseu_1}
\end{center}
\end{figure} 

%%%% FIGURE %%%

\begin{figure}[http]
\begin{center}
\begin{minipage}{0.49\textwidth}
\includegraphics[width=\textwidth]{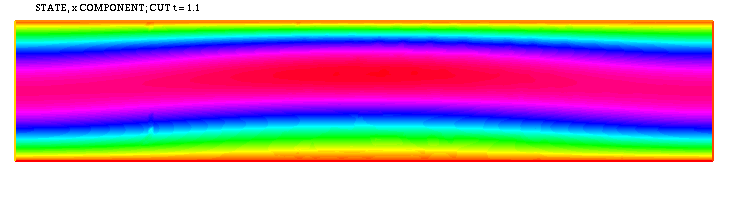}
\includegraphics[width=\textwidth]{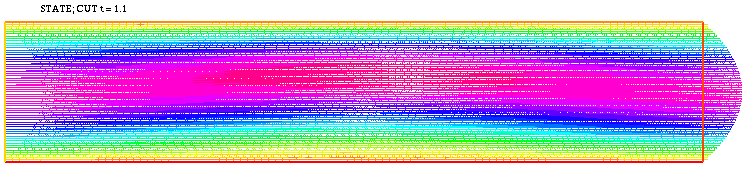}
\end{minipage}
\begin{minipage}{0.49\textwidth}
\includegraphics[width=\textwidth]{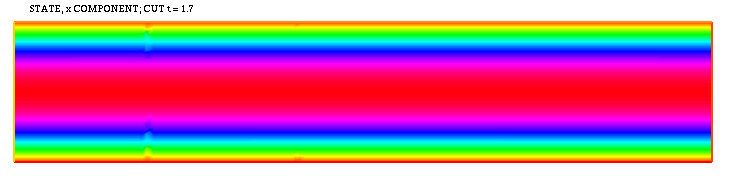}
\includegraphics[width=\textwidth]{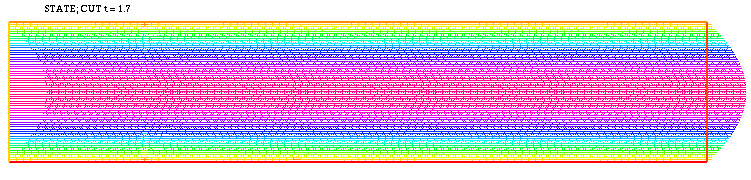}
\end{minipage}
\caption{Poiseuille test -- The state at $T=1.1$ ({\bf Left}) and 
the state at $T=1.7$ (({\bf Right}).}
\label{fig_poiseu_2}
\end{center}
\end{figure} 

%%%% FIGURE %%%

\begin{figure}[http]
\begin{center}
\begin{minipage}{0.7\textwidth}
\includegraphics[width=\textwidth]{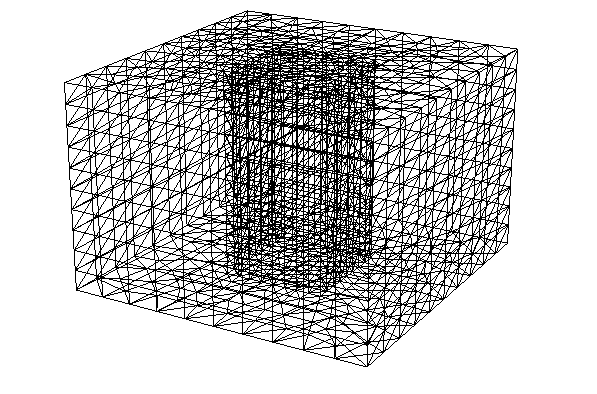}
\end{minipage}
\caption{Taylor-Green test -- The domain and the mesh. Number of vertices: 3146. Number of elements (tetrahedra): 15900.
Total number of variables: 22022.}
\label{mesh_taylor}
\end{center}
\end{figure}

%%%% FIGURE %%%

\begin{figure}[http]
\begin{center}
\begin{minipage}{0.49\textwidth}
\includegraphics[width=\textwidth]{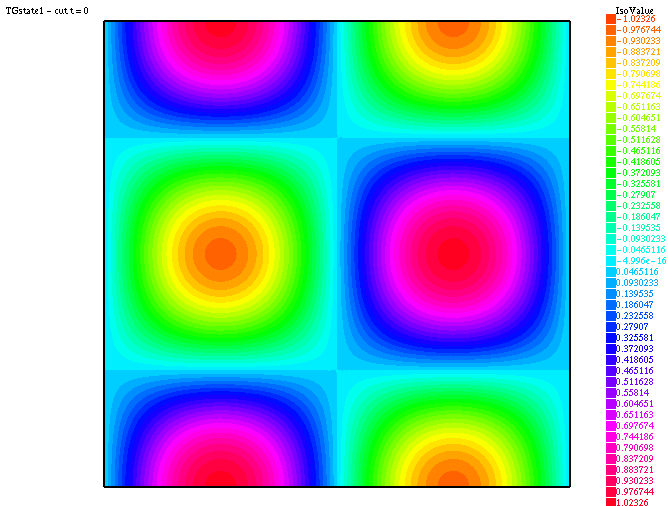}
\includegraphics[width=\textwidth]{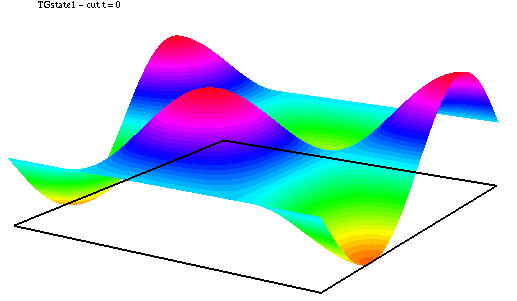}
\end{minipage}
\begin{minipage}{0.49\textwidth}
%\hspace*{0.5cm}
\includegraphics[width=\textwidth]{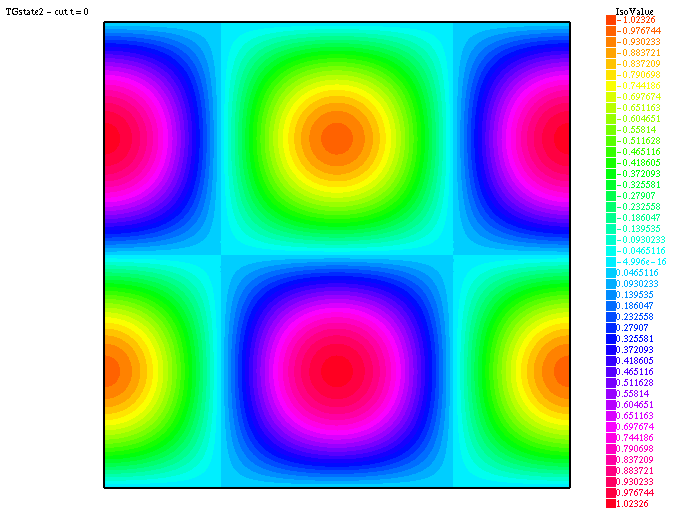}
\includegraphics[width=\textwidth]{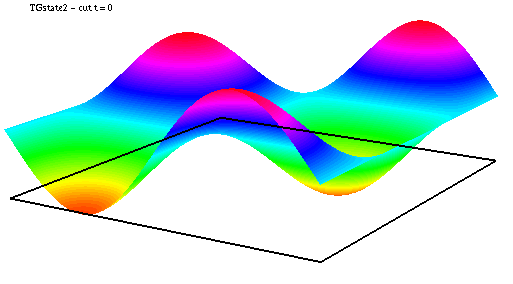}
\end{minipage}
\caption{Taylor-Green test -- First component of the initial datum ({\bf Left}) and 
second component of the initial datum ({\bf Right}).}
\label{fig_taylor_1}
\end{center}
\end{figure} 

%%%% FIGURE %%%

\begin{figure}[http]
\begin{center}
\begin{minipage}{0.45\textwidth}
\includegraphics[width=\textwidth]{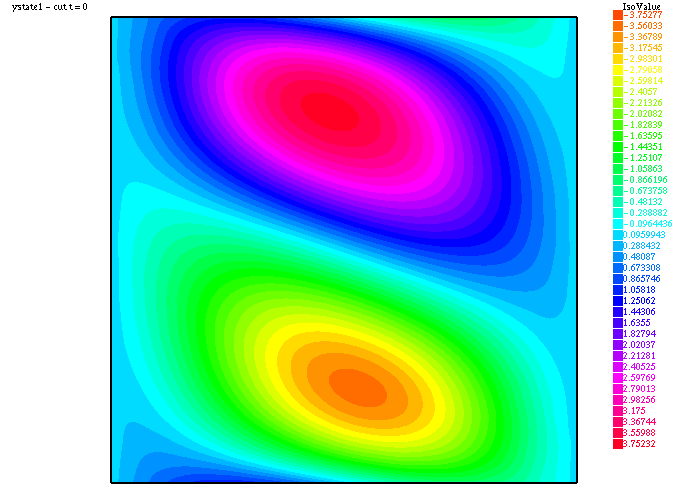}
\includegraphics[width=\textwidth]{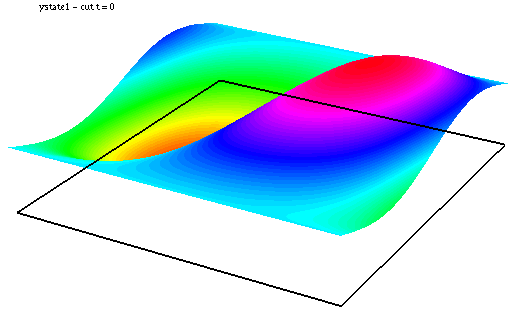}
\end{minipage}
\begin{minipage}{0.45\textwidth}
%\hspace*{0.5cm}
\includegraphics[width=\textwidth]{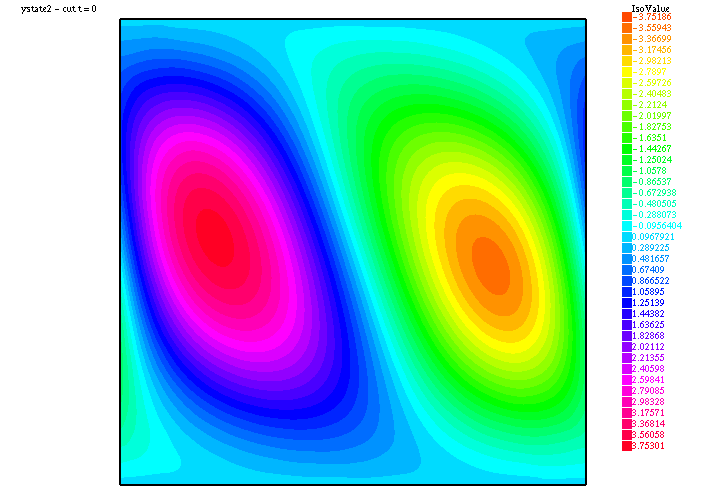}
\includegraphics[width=\textwidth]{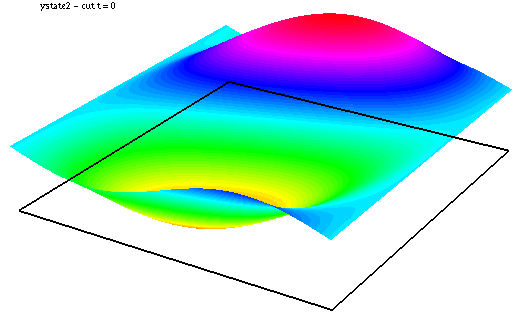}
\end{minipage}
\caption{Taylor-Green test -- The initial data: first component ({\bf Left}) and 
second component ({\bf Right}).}
\label{fig_taylor_2}
\end{center}
\end{figure} 

%%%% FIGURE %%%

\begin{figure}[http]
\begin{center}
\begin{minipage}{0.45\textwidth}
\includegraphics[width=\textwidth]{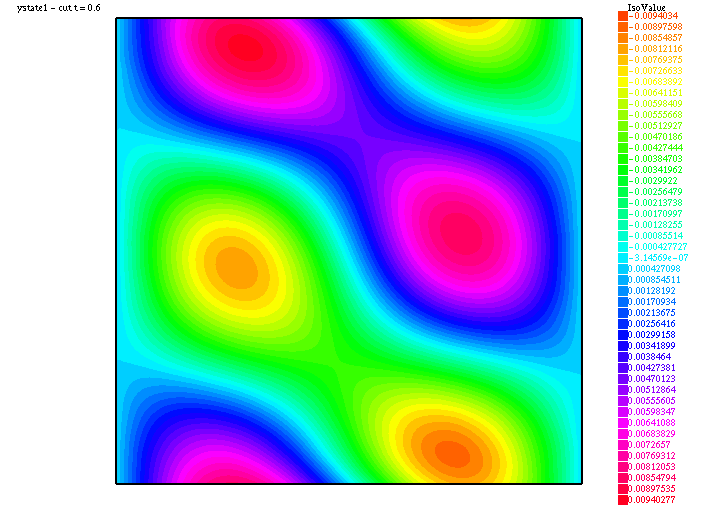}
\includegraphics[width=\textwidth]{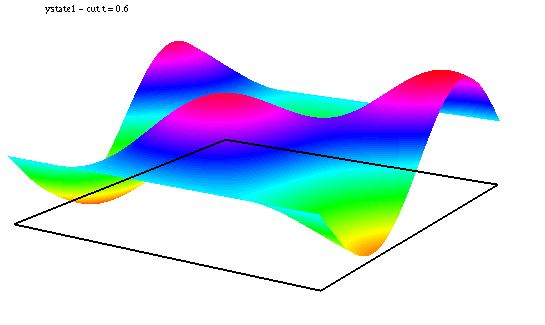}
\end{minipage}
\begin{minipage}{0.45\textwidth}
\includegraphics[width=\textwidth]{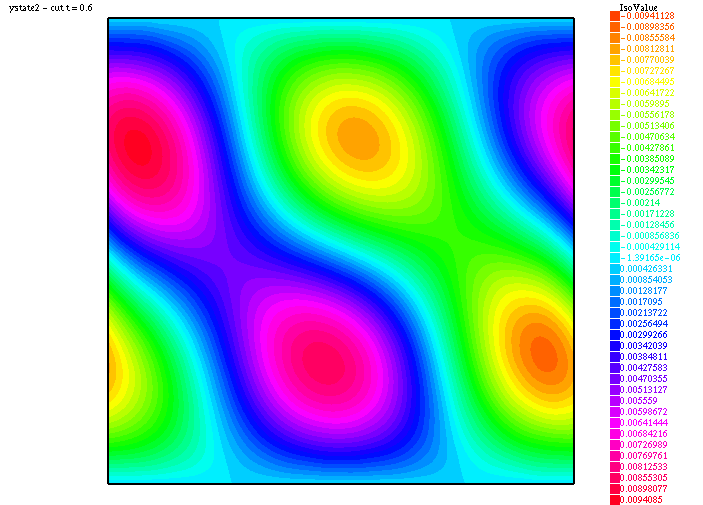}
\includegraphics[width=\textwidth]{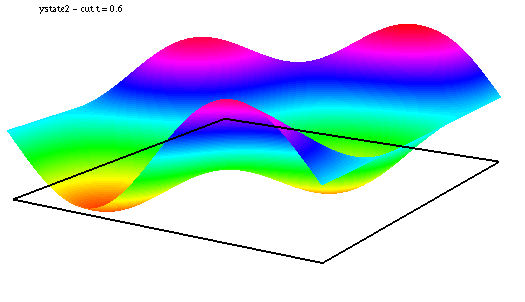}
\end{minipage}
\caption{Taylor-Green test -- The first component of the state  ({\bf Left}) and 
the second component of the  state at $T=0.6$ (({\bf Right}).}
\label{fig_taylor_3}
\end{center}
\end{figure} 

%%%% FIGURE %%%

\begin{figure}[http]
\begin{center}
\begin{minipage}{0.45\textwidth}
\includegraphics[width=\textwidth]{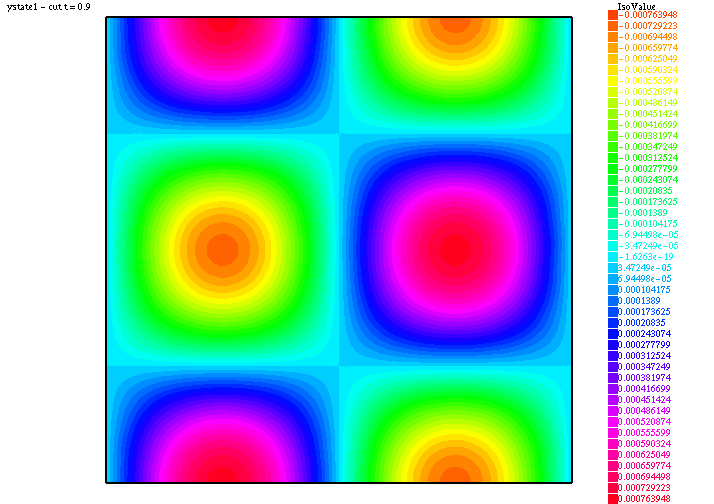}
\includegraphics[width=\textwidth]{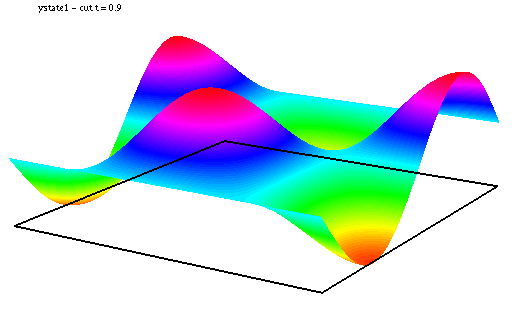}
\end{minipage}
\begin{minipage}{0.45\textwidth}
\includegraphics[width=\textwidth]{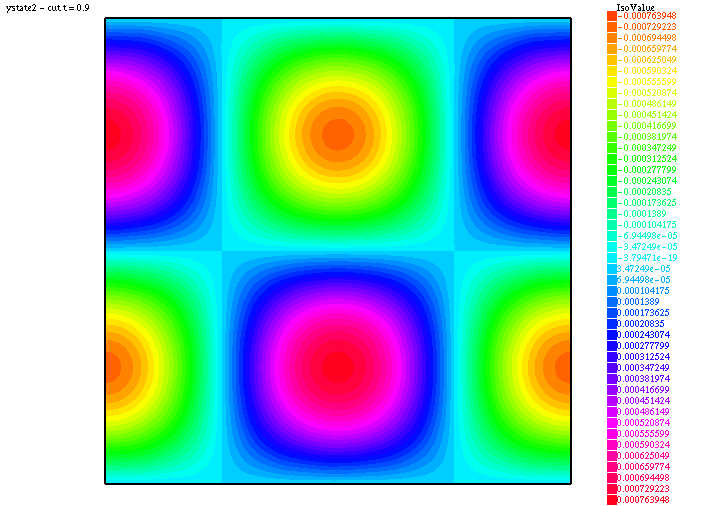}
\includegraphics[width=\textwidth]{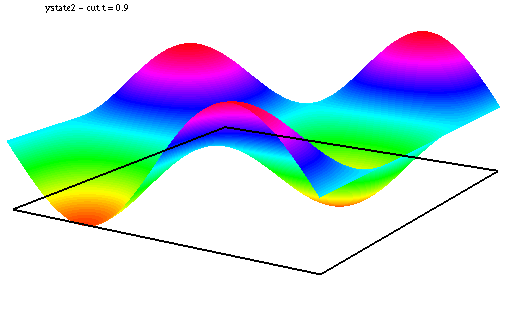}
\end{minipage}
\caption{Taylor-Green test -- The first component of the state ({\bf Left}) and 
the second component of the  state at $T=0.9$ (({\bf Right}).}
\label{fig_taylor_4}
\end{center}
\end{figure} 

%%%% FIGURE %%%

\begin{figure}[http]
\vskip-3cm
\begin{center}
\begin{minipage}{0.5\textwidth}
\includegraphics[width=\textwidth]{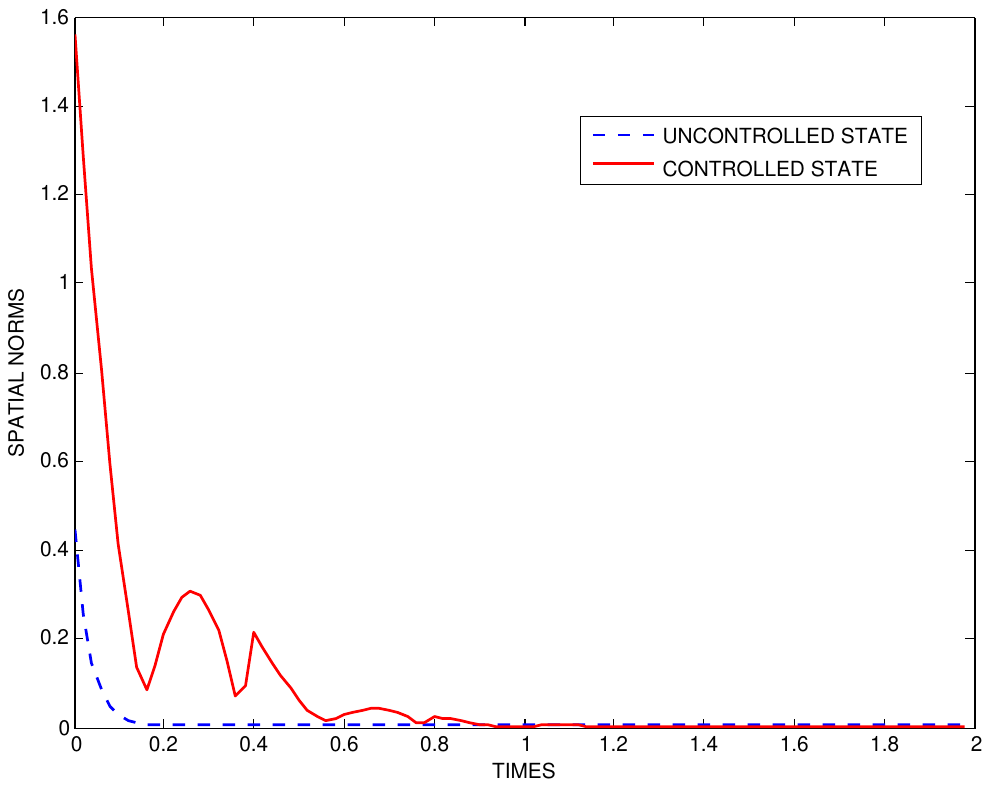}
\end{minipage}
\hskip-.5cm
\begin{minipage}{0.5\textwidth}
\includegraphics[width=\textwidth]{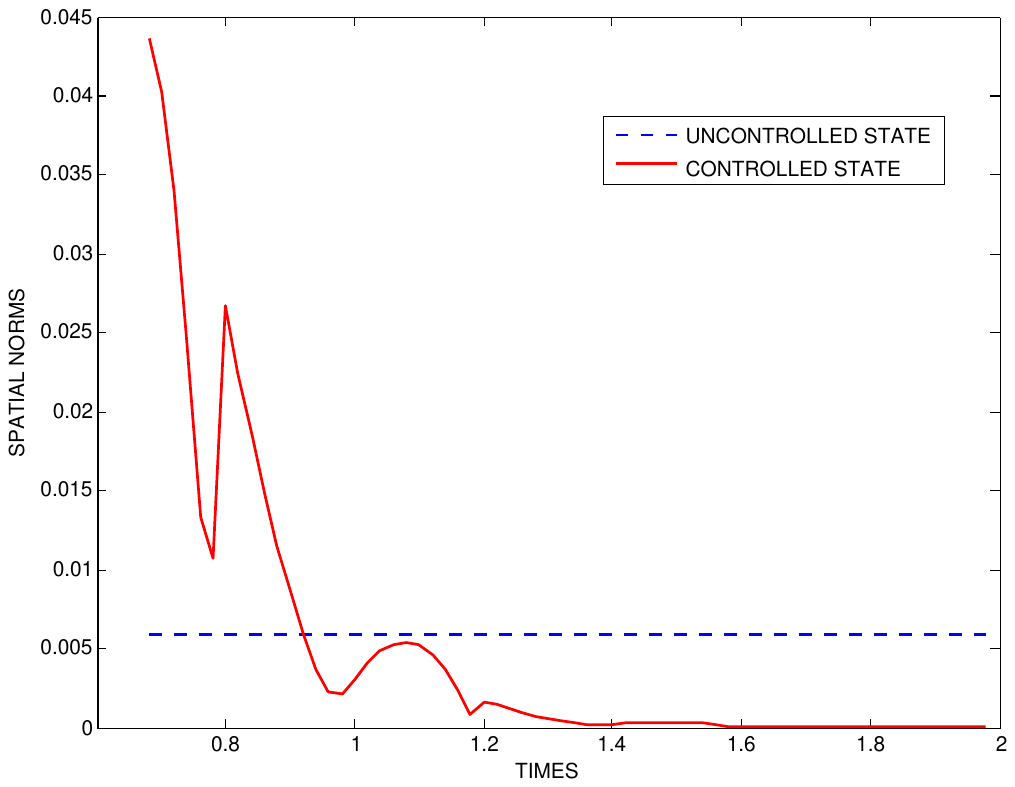}
\end{minipage}
\vskip-3cm
\caption{Poiseuille test -- The $L^2$-norms of the deviations $\yvec - \overline{\yvec}_P$ and~$\zvec - \overline{\yvec}_P$ corresponding to the computed controlled and uncontrolled states ({\bf Left}) and a detail (({\bf Right}).
   Here, $T=2$.
   For instance, at $t=1.98$ we have $\|\yvec(\cdot\,,t) - \overline{\yvec}_P\|_{\Lvec^2(\Om)} = 1.45367 \cdot 10^{-7}$ and~$\|\zvec(\cdot\,,t) - \overline{\yvec}_P\|_{\Lvec^2(\Om)} = 0.00597072$.}\label{snormP}
\end{center}
\end{figure} 

%%%% FIGURE %%%

\begin{figure}[http]
\vskip-3cm
\begin{center}
\begin{minipage}{0.5\textwidth}
\includegraphics[width=\textwidth]{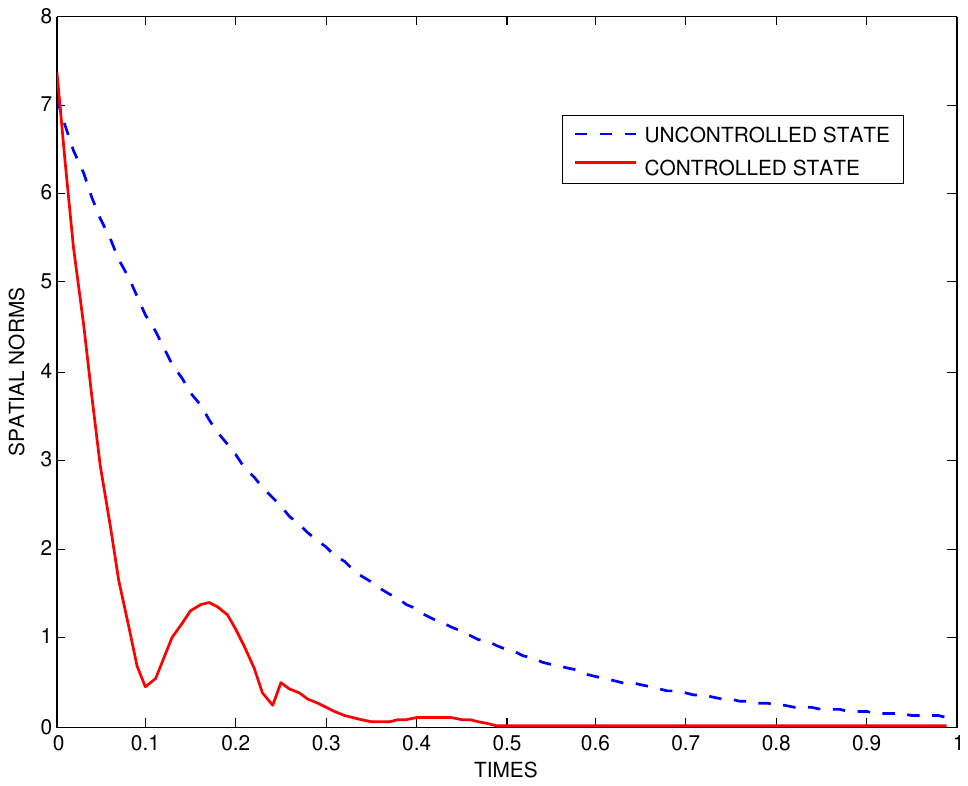}
\end{minipage}
\hskip-.5cm
\begin{minipage}{0.5\textwidth}
\includegraphics[width=\textwidth]{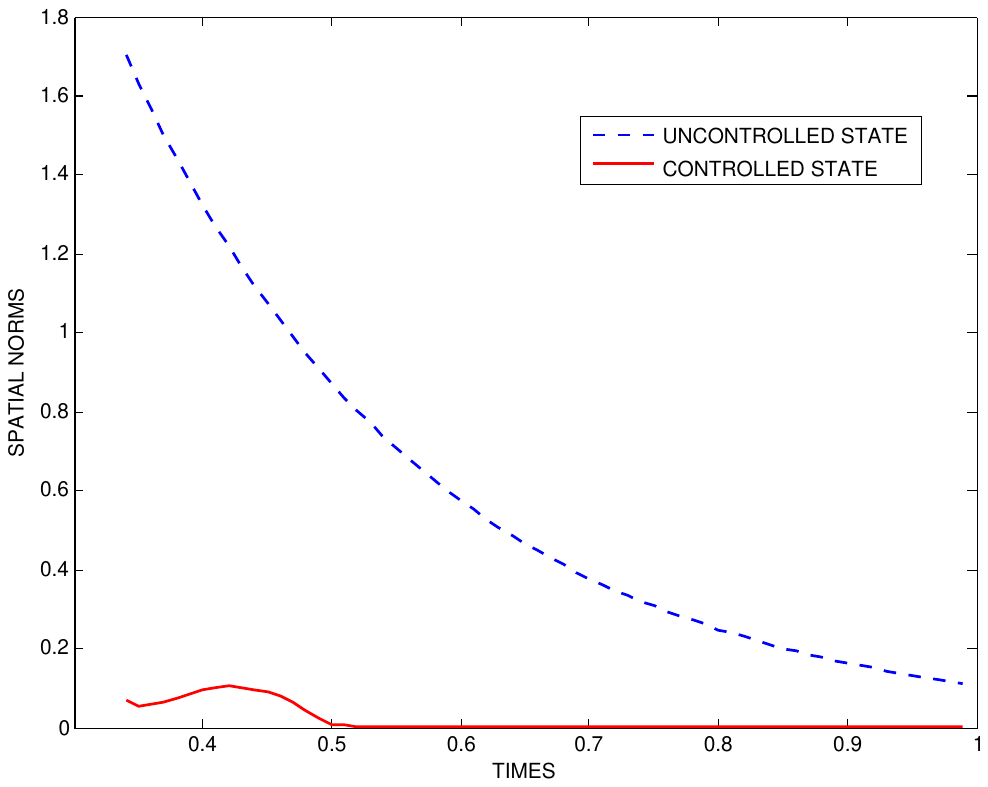}
\end{minipage}
\vskip-3cm
\caption{Taylor-Green test -- The $L^2$-norms of the deviations $\yvec - \overline{\yvec}_{TG}$ and~$\zvec - \overline{\yvec}_{TG}$ corresponding to the computed controlled and uncontrolled states ({\bf Left}) and a detail (({\bf Right}).
   Here, $T=1$.
   For instance, at $t=0.99$ we have $\|\yvec(\cdot\,,t) - \overline{\yvec}_{TG}(\cdot\,,t)\|_{\Lvec^2(\Om)} = 3.74645 \cdot 10^{-10}$ and~$\|\zvec(\cdot\,,t) - \overline{\yvec}_{TG}(\cdot\,,t)\|_{\Lvec^2(\Om)} = 0.112909$.}
\label{snormTG}
\end{center}
\end{figure} 

%%%%%%%%%%%%%%%%%%%%%%%%%%%%%%%%%%
%%%%  NEW SECTION 
%%%%%%%%%%%%%%%%%%%%%%%%%%%%%%%%%%

%%%% ATTENTION: GOOD CONVERGENCE, ITERATES VERSUS $M$, CONCLUSIONS? MORE REFERENCES?

\section{Additional comments and conclusions}\label{Sec-final}

   In this paper, we have seen that it is possible to solve numerically null controllability problems for the two-dimensional heat, Stokes and Navier-Stokes equations  with Dirichlet boundary conditions.
   We have used some ideas that come from the so called Fursikov-Imanuvilov formulation and lead to the solution of high order partial differential problems in the space and time variables.
   The similar and simpler one-dimensional case was studied in~\cite{EFC-AM-sema}.
   
   There are two different ways to define numerical approximations of the resulting systems:
   
\begin{itemize}

\item By working with spatially $C^1$ finite element spaces or 

\item By introducing multipliers and working on (usual) $C^0$ finite element spaces.

\end{itemize}

   In this paper, we have chosen the second approach.
   In a forthcoming paper, we will be concerned with the first one.
   
   Unfortunately, in our case, the numerical approximation is not completely justified from a rigorous mathematical viewpoint.
   However, we have seen that the approximate problems can be solved in a relatively easy way and produce good numerical results.
   
   Of course, the same ideas and techniques can be applied in many other similar situations:
   semilinear heat equations with (for instance) globally Lipschitz-continuous nonlinearities, Boussinesq-like systems, non-cylindrical control domains, boundary control problems, etc.
   Some results have been presented in~\cite{NC-EFC-AM, EFC-AM-mcrf};
   see also~\cite{EFC-AM-dual}.	

%%%% REFERENCES

%\bibliographystyle{siam}
%\bibliography{biblio}

\end{document}